\documentclass[11pt,a4paper]{article}
 \usepackage{euscript}

\usepackage{amsmath}
\usepackage{graphicx}
\usepackage{amsthm}
\usepackage{amssymb}
\usepackage{epsfig}

\usepackage{url}

\theoremstyle{theorem}
\newtheorem{theorem}{Theorem}[section]
\newtheorem{corollary}[theorem]{Corollary}

\newtheorem{lemma}[theorem]{Lemma}
\newtheorem{proposition}[theorem]{Proposition}

\newtheorem{definition}{Definition}[section]

\numberwithin{equation}{section}

\title{Self-dual Yang-Mills Equations in  Split Signature}
 \date{\today}
\author{Masood Aryapoor\footnote{masood.aryapoor@yale.edu}\\ \\
Mathematics Department\\
Yale University\\
442 Dunham Lab\\
10 Hillhouse Avenue\\
New Haven, CT 06511 USA\\}
\begin{document}  
\maketitle
\begin{abstract}
We study the self-dual Yang-Mills equations in split signature. We give a special solution, called the basic split instanton, and describe the ADHM construction in the split signature. Moreover a split version of t'Hooft ansatz is described. 
\end{abstract}
\begin{section}{Introduction}
Self-dual Yang-Mills equations (SDYM for short) are well-known equations in dimension $4$.  They were introduced in the last century and  many beautiful applications of them  in other areas of mathematics have been found  since then, for their applications in four dimensional geometry see \cite{DK,FU}.  \\
These equations are defined using a metric on a four dimensional manifold. It turns out that there are two possible choices for the signature of the metric to obtain real valued solutions, namely Euclidean and split signatures.  
Much of the research has been focused on the Euclidean case. In this paper we would like to study these equations in the split signature, for their applications  in integrable systems see \cite{MW}.\\  
 The starting point for us to study the split SDYM equations is the existence of a very special solution in the split signature which is quite similar to  the basic instanton.  The existence of this solution shows that the split SDYM equations might be as important as they are in the Euclidean case.  On the other hand the split SDYM equations  are not as rigid as the ones in the Euclidean signature. More precisely, 
as it is well-known, the SDYM equations in the Euclidean signature are of elliptic type which confirms the finite dimensionality of the moduli space. In contrast to the Euclidean signature, in the split signature, the equations are not of elliptic type and one cannot hope to have a finite dimensional moduli space of solutions. In fact L.J. Mason has recently shown that there is a one to one  correspondence between SDYM solutions in the split signature over $S^2\times S^2$ and certain data on the complex projective space $\mathbb{C}\mathbb{P}^3$, see \cite{Ma}. The data consists of a holomorphic and a smooth part which shows that the moduli space cannot be finite dimensional. Therefore in the split case, we have a lot more freedom. But in order to obtain a finite dimensional moduli space of solutions in the split case we have to impose some conditions on the set of solutions. For the simplest case, i.e. when the structure group is $O(2)$ and the "charge" is $\pm 1$ we introduce an extra condition on the solutions. With this extra condition, the moduli space becomes isomorphic to $SO_0(3,3)/SO(3)\times SO(3)$ in parallel with  the classical result that the moduli space of instantons of charge 1 is isomorphic to $SO_0(5,1)/SO(5)$. \\
Another analogy between these two signatures is the existence of Atiyah-Drinfeld-Hitchin-Manin (ADHM for short) construction in the split case. It is well-known that all the solutions of SDYM equations on $S^4$ are given by the ADHM construction which is basically an algebraic construction, see \cite{ADHM}. In fact it is easy to show that this construction produces solutions but it is much harder to prove that the ADHM construction yields all the solutions.
It turns out that  there is an analogous construction for the solutions of SDYM equations in the split signature. More precisely we give a construction similar to the ADHM construction which produces solutions of  SDYM equations on  the conformal compactification of $\mathbb{R}^{2,2}$ with the split signature metric $ds^2=dx_1^2+dx_2^2-dx_3^2-dx_4^2$. \\  
Another motivation to study the split SDYM equations comes from Representation Theory.  One can use the moduli space of instantons to realize a family of  representations of certain infinite dimensional Lie algebras geometrically, see \cite{Na} and references therein. 
 However, for this application, one has to "compactify" the moduli space of instantons in a suitable manner and consider the so-called "ideal" instantons. It would be nice to realize ideal instantons  as genuine solutions! Because there is no room in the Euclidean picture, one is led to consider  the split SDYM equations. As we will see, there is a family of $O(2)$-SDYM solutions in the split case whereas there is no non-trivial Euclidean $U(1)$-instanton. \\
Here is an outline of the paper . In the first part,    we first review the construction of the basic instanton and then we introduce the basic split instanton. Finally we explain a close relation between these two special solutions in Euclidean and split signatures.\\
In the second part, we deal with the moduli problem in the split case.  We show that the whole moduli space of $O(2)$-SDYM solutions of topological charge $1$ is infinite dimensional. But if we just consider the solutions which have a large symmetry group, then the restricted moduli space is finite dimensional and is isomorphic to $SL(4,\mathbb{R})/SO(4)$.\\ 
 In the third part, we introduce the split t'Hooft ansatz. Historically, the first SDYM solutions in the Euclidean signature were given by the so-called t'Hooft ansatz. This ansatz starts with a solution $f$ of the Laplacian equation on some region in $S^4$ and construct a solution of the Euclidean SDYM equations, say $A$. A priori, the solution $A$ is nonsingular on the same region, but it could happen that the solution $A$ has less singularity than $f$. In fact since there is no nontrivial solution to the Laplacian equation on the 4-sphere, the only way to construct global solutions from the t'Hooft ansatz is to start from the local solutions of  the Laplacian equation and hope that the ansatz gives a global solution. For this reason we refer to this ansatz as the local t'Hooft ansatz. 
We show that we also have a version of the t'Hooft ansatz in the split signature as well. One only needs to start from a solution to the ultra-hyperbolic equation. Since in the split signature there is plenty of global solutions, given any global solution of the ultra-hyperbolic equation,    we can construct a global ASDYM solution. We call this ansatz, the global t'Hooft ansatz. The global t'Hooft ansatz only produces $GL(2,\mathbb{R})$-ASDYM solutions. Some of these solutions are in fact $O(2)$ solutions. We also show that, all the anti-instantons (defined in part two) can be obtained via the global t'Hooft ansatz.\\
Finally, in the last part, we present the split ADHM construction. It turns out that there is a complex version of the ADHM construction as well. The complex ADHM construction gives rise to holomorphic vector bundles on $Gr(2,\mathbb{C}^4)$. It is an analog of "monad" construction of holomorphic vector bundles on complex projective spaces, see \cite{OS}.  \\ \\
\noindent \textbf{Acknowledgment:} I would like to thank professor Igor Frenkel without whom this paper would not be done. The author is very grateful to him for suggesting this project, supports, encouragement  and 
very useful discussions. I would also like to thank \mbox{Professors} M. Kapranov and G. Zuckerman for helpful and informative conversations.

\end{section} 

\begin{section}{Basic Euclidean and Split Instantons}
 There is a very special $SU(2)$-SDYM solution on $S^4$ which is called the basic instanton.   It has a very nice description using the algebra of quaternions. It turns out that there is a special $O(2)$-SDYM solution in the split case as well which is quite similar to the basic instanton. It also has a very nice description in terms of the algebra of split quaternions. We call this solution the basic split instanton by analogy. In this section we review the construction of the basic instanton and describe the basic split instanton. Moreover we show that there is a surprising relation between the basic instanton and the basic split anti-instanton by passing to the complex picture.
\begin{subsection}{Preliminaries}
By the $\mathbb{C}$-algebra of complex quaternions, denoted by $\mathbb{H}_{\mathbb{C}}$, we just mean the $\mathbb{C}$-algebra of two by two complex matrices. So
$$\mathbb{H}_{\mathbb{C}}:=\{ Z=\begin{pmatrix}
z_{11} & z_{12}\\
z_{21} & z_{22}
\end{pmatrix} | z_{ij}\in\mathbb{C} \} $$
Given $Z=\begin{pmatrix}
z_{11} & z_{12}\\
z_{21} & z_{22}\end{pmatrix}$, we set
$$Z^t:=\begin{pmatrix}
z_{11} & z_{21}\\
z_{12} & z_{22}\end{pmatrix}$$
$$\tilde{Z}:=\begin{pmatrix}
z_{22} & -z_{12}\\
-z_{21} & z_{11}\end{pmatrix}$$
$$Z^*:=\begin{pmatrix}
\bar{z}_{11} & \bar{z}_{21}\\
\bar{z}_{12} & \bar{z}_{22}\end{pmatrix}$$
Identifying  $\mathbb{H}_{\mathbb{C}}$ with $\mathbb{C}^4$ as a complex manifold,  we consider the following (holomorphic) metric and volume form on $\mathbb{H}_{\mathbb{C}}$
$$ds^2:=2(dz_{11}dz_{22}-dz_{12}dz_{21})=2\det{dZ}$$
$$dV=dz_{11}\wedge dz_{21}\wedge dz_{12}\wedge dz_{22}$$
 Therefore we have the Hodge $*$-operator 
$$*:\Omega^2(\mathbb{H}_{\mathbb{C}})\to\Omega^2(\mathbb{H}_{\mathbb{C}})$$
where $\Omega^n(\mathbb{H}_{\mathbb{C}})$ is the sheaf of holomorphic $n$-forms on $\mathbb{H}_{\mathbb{C}}$. We recall that for 2-forms $\alpha$ and  $ \beta$ we have
$$\alpha\wedge *\beta=(\alpha,\beta)dV$$
It is easy to see that $*^2=1$.
A 2-form $\omega$ on $\mathbb{H}_{\mathbb{C}}$ is called self-dual (or SD for short) if $*\omega=\omega$ and it is called anti-self-dual (or ASD) if $*\omega=-\omega$.
The space of SD 2-forms is generated by 
$$dz_{11}\wedge dz_{21},dz_{12}\wedge dz_{22},  dz_{11}\wedge dz_{22}+dz_{12}\wedge dz_{21}$$
and the  space of ASD 2-forms is generated by
$$dz_{11}\wedge dz_{12}, dz_{21}\wedge dz_{22},  dz_{11}\wedge dz_{22}-dz_{12}\wedge dz_{21}$$
We have the following simple algebraic lemma concerning SD and ASD 2-forms on $\mathbb{H}_{\mathbb{C}}$
\begin{lemma}\label{AL}
Let $A\in \mathbb{H}_{\mathbb{C}}$. Then \\
(a) The $\mathbb{H}_{\mathbb{C}}$-valued 2-form $dZ\wedge AdZ^t$ is SD if and only if $A=A^t$.\\
(b) The $\mathbb{H}_{\mathbb{C}}$-valued 2-form $dZ^t\wedge AdZ$ is ASD if and only if $A=A^t$.\\
(c)  The $\mathbb{H}_{\mathbb{C}}$-valued 2-form $dZ\wedge Ad\tilde{Z}$ is ASD if and only if $A\in \mathbb{C}$.\\
(d)   The $\mathbb{H}_{\mathbb{C}}$-valued 2-form $d\tilde{Z}\wedge AdZ$ is SD if and only if $A\in \mathbb{C}$.\\ 
where $dZ=\begin{pmatrix}
dz_{11} & dz_{12}\\
dz_{21} & dz_{22}\end{pmatrix}$ and the same for $dZ^t$ and $d\tilde{Z}$.
\end{lemma}
Having defined SD and ASD 2-forms, we can consider the self-dual and anti-self-dual Yang-Mills equations (or SDYM and ASDYM equation for short) on $\mathbb{H}_{\mathbb{C}}$.  The solutions are (holomorphic) connections (defined on a holomorphic vector bundle on 
$\mathbb{H}_{\mathbb{C}}$) whose curvature is SD or ASD.  We briefly recall the notion of connection and the related concepts, for the details see \cite{KN} for example.  A connection on   a holomorphic vector bundle $V$ on $\mathbb{H}_{\mathbb{C}}$ is a $\mathbb{C}$-linear morphism 
$$\nabla:\mathcal{O}_V\to\Omega^1(V)$$
which satisfies 
$$\nabla(fs)=\partial f\otimes s+f\nabla s$$
for any holomorphic function $f$ and a holomorphic section $s$ of $V$.  Here $\mathcal{O}_V$ is the sheaf of holomorphic sections of $V$ and $\Omega^n(V)$ is the sheaf of holomorphic $V$-valued $n$-forms. In the local frame $u=(s_1,...,s_n)$ of the vector bundle $V$ of rank $n$, a connection can be written as $\partial+A$ where $A$ is an $n$ by $n$  matrix of 1-forms called the connection potential of $\nabla$ in the local frame $u$.  More precisely, for an $n$ by 1 matrix $f$ of holomorphic functions  we have
$$\nabla(uf)=u(\partial f+Af)$$
If $u_1$ is another local frame for $V$, then $u_1=ug$ for some  $g$, called a gauge transformation, which is an $n$ by $n$ matrix of holomorphic functions. Then it is easy to see that the connection potential of $\nabla$ in this new local frame is given by
$$g^{-1}\partial g+g^{-1}Ag$$
It is well-known that $\nabla$ has a natural extension
$$\nabla_1:\Omega^1(V)\to \Omega^{2}(V)$$
defined by $\nabla_1(s\otimes\alpha)=\nabla(s)\wedge\alpha+s\otimes \partial\alpha$ where $s$ is a section of $\mathcal{O}_V$ and $\alpha$ is a holomorphic 2-form. The curvature $C$ of this connection is defined to be $\nabla_1\circ \nabla$. It is easy to see  that $C$ is a bundle homomorphism and hence it defines a section of $\Omega^2(End(V))$. In the local frame $u$, the curvature is given by an $n$ by $n$ matrix of 2-forms $F$, called the connection 2-form, 
$$F=\partial A+A\wedge A$$
Moreover under the  gauge transformation $g$, $F$ is transformed to $g^{-1}Fg$.\\ 
The SDYM and ASDYM equations are $*C=C$ and $*C=-C$ respectively where    $\nabla$ is a connection defined on a vector bundle $V$ and $C$ is its curvature. If 
 $$A=A_{11}dz_{11}+A_{12}dz_{12}+A_{21}dz_{21}+A_{22}dz_{22}$$
is the connection potential of $\nabla$ in some local gauge, then the SDYM equations are   
$$F_{11,12}=F_{21,22}=F_{11,22}-F_{12,21}=0$$
where $F_{ij,kl}:=\frac{\partial{A_{kl}}}{\partial z_{ij}}-\frac{\partial{A_{ij}}}{\partial z_{kl}}+[A_{ij},A_{kl}]$. Similarly the ASDYM equations are
$$F_{11,21}=F_{12,22}=F_{11,22}+F_{12,21}=0$$
Obviously the whole discussion so far has a counterpart in the category of smooth manifolds which we use as well.
There are two real forms of $\mathbb{H}_{\mathbb{C}}$ that we are interested in. The first one is the Euclidean real form
$$\mathbb{H}:=\{ \begin{pmatrix}
z_{11} & z_{12}\\
z_{21} & z_{22}
\end{pmatrix} \in\mathbb{H}_{\mathbb{C}} |  z_{22}=\bar{z}_{11},\; z_{12}=-\bar{z}_{21} \} $$
which is the $\mathbb{R}$-algebra of quaternions. The other one is the split real form
$$\mathbb{H}_{\mathbb{R}}:=\{ \begin{pmatrix}
z_{11} & z_{12}\\
z_{21} & z_{22}
\end{pmatrix} \in\mathbb{H}_{\mathbb{C}}| z_{ij}\in\mathbb{R} \} $$
which is the $\mathbb{R}$-algebra of split quaternions. we have the following simple lemma concerning these real forms of $\mathbb{H}_{\mathbb{C}}$
\begin{lemma}\label{AL1}
Let $A\in \mathbb{H}_{\mathbb{C}}$. Then \\
(a) $A\in\mathbb{H}$ if and only if $A^*=\tilde{A}$\\
(b) $A\in \mathbb{H}_{\mathbb{R}}$ if and  only if $A^*=A^t$.
\end{lemma}
The restrictions of the metric of $\mathbb{H}_{\mathbb{C}}$ to these real forms have different signatures. As the names suggest, the restriction of the metric to $\mathbb{H}$ is Euclidean and its restriction to $\mathbb{H}_{\mathbb{R}}$ has the split signature $(+,+,-,-)$.  We can also consider the restriction of SDYM and ASDYM equation on these real form.  
We are interested in SDYM and ASDYM equations on these real forms and their corresponding conformal compactifications which we introduce next.  Consider 
 $$S:=\{(Z,W)\in \mathbb{H}_{\mathbb{C}}^2|\; \nexists A\in \mathbb{H}_{\mathbb{C}}\setminus \{0\} \quad  \text{s.t.} \quad ZA=WA=0\}$$
 Two elements $(Z,W), (Z_1,W_1)\in S$ are called equivalent if there is  an invertible element $q\in  \mathbb{H}_{\mathbb{C}}$ such that $Z_1=Zq, W_1=Wq$. Let $ \mathbb{H}_{\mathbb{C}}\mathbb{P}^1$, called the complex quaternionic projective line, be the set of equivalence classes of elements in $S$. The equivalence class of $(Z,W)\in S$ in 
$ \mathbb{H}_{\mathbb{C}}\mathbb{P}^1$  is denoted by $[Z:W]$. It is easy to see that $ \mathbb{H}_{\mathbb{C}}\mathbb{P}^1$ is a complex manifold isomorphic to $Gr(2,\mathbb{C}^4)$, the Grassmannian of complex 2-planes in $\mathbb{C}^4$. One isomorphism is given by sending $[Z:W]$ to the 2-plane generated by 
$$\begin{pmatrix}
z_{11}\\
z_{21}\\
w_{11}\\
w_{21}
\end{pmatrix}\; \text{and}\; \begin{pmatrix}
z_{12}\\
z_{22}\\
w_{12}\\
w_{22}
\end{pmatrix}$$
where 
 $$Z=\begin{pmatrix}
z_{11} & z_{12}\\
z_{21} & z_{22}
\end{pmatrix}\;\text{and}\;W=\begin{pmatrix}
w_{11} & w_{12}\\
w_{21} & w_{22}
\end{pmatrix}
$$
We identify $\{[Z:1]|Z\in \mathbb{H}_{\mathbb{C}}\}$ with $\mathbb{H}_{\mathbb{C}}$. Then one can see that the conformal structure on $\mathbb{H}_{\mathbb{C}}$ (given by the metric $ds^2$ as above) and its volume form extend to   $ \mathbb{H}_{\mathbb{C}}\mathbb{P}^1$. Hence we can consider the ASDYM and SDYM equations on $ \mathbb{H}_{\mathbb{C}}\mathbb{P}^1$ as the extensions of those on $ \mathbb{H}_{\mathbb{C}}$ because these equations are conformally invariant in dimension four. \\
In the similar way we can define the quaternionic and split quaternionic projective lines which we denote by $ \mathbb{H}\mathbb{P}^1$ and $ \mathbb{H}_{\mathbb{R}}\mathbb{P}^1$ respectively. 
These are totally real sub-manifolds of $ \mathbb{H}_{\mathbb{C}}\mathbb{P}^1$ of real dimension four. It is easy to see that 
$$ \mathbb{H}\mathbb{P}^1\cong S^4\; \text{and}\;  \mathbb{H}_{\mathbb{R}}\mathbb{P}^1\cong Gr(2,\mathbb{R}^4)$$
as smooth manifolds. Here $Gr(2,\mathbb{R}^4)$ is  the Grassmannian of real 2-planes in $\mathbb{R}^4$. 
In this paper we mainly deal with SDYM and ASDYM equations on $\mathbb{H}_{\mathbb{R}}\mathbb{P}^1$. \\ \\
Finally we explain the notion of $G$-SDYM and $G$-ASDYM solutions.  If $G$ is a Lie group, then by a $G$-SDYM (or $G$-ASDYM) solution we mean a vector bundle with a $G$-structure and an SDYM (or ASDYM) connection compatible with the $G$-structure.  \\   
\end{subsection}

  \begin{subsection}{Basic (Euclidean) instanton}
The (A)SDYM equations in the Euclidean case have been studied intensively over the past years.  
Here we just review some of the basic facts in the Euclidean case, see \cite{At}.\\   We identify $\mathbb{H}$ with  $\mathbb{R}^4$ via 
$$x:=(x_1,x_2,x_3,x_4)\mapsto\begin{pmatrix}
x_1+ix_4&x_2+ix_3\\
-x_2+ix_3&x_1-ix_4
\end{pmatrix}$$
We see that under this isomorphism the operator $Z\to\tilde{Z}$ becomes $$x\mapsto\bar{x}:=(x_1,-x_2,-x_3,-x_4)$$
In these new coordinates, the metric is just the Euclidean  metric  $ds^2=2(dx_1^2+dx_2^2+dx_3^2+dx_4^2)$ and the volume form is $dV=-4dx_1\wedge dx_2\wedge dx_3\wedge dx_4. $   It is easy to see that  the space of SD 2-forms is spanned by the following two forms
$$ dx_1\wedge dx_2-dx_3\wedge dx_4,$$
$$dx_1\wedge dx_3+dx_2\wedge dx_4,$$
 $$dx_1\wedge dx_4-dx_2\wedge dx_3$$
and the space of ASD 2-forms is spanned by the following 2-forms
 $$dx_1\wedge dx_2+dx_3\wedge dx_4,$$
$$dx_1\wedge dx_3-dx_2\wedge dx_4,$$
 $$dx_1\wedge dx_4+dx_2\wedge dx_3$$
It is easy to see that the SDYM equations on the connection potential $A=A_1dx_1+A_2dx_2+A_3dx_3+A_4dx_4$ of a connection (where $A_i:\mathbb{R}^4\to g$ are smooth functions and $g$ is the Lie algebra of some Lie group $G$) are equivalent to  the  following system of partial differential equations: 
$$F_{12}+F_{34}=F_{13}-F_{24}=F_{14}+F_{23}=0$$
 where $F_{ij}=\frac{\partial A_j}{\partial x_i}-\frac{\partial A_i}{\partial x_j}+[A_i,A_j]$. Also, the ASDYM equations are given by   $F_{12}-F_{34}=F_{13}+F_{24}=F_{14}-F_{23}=0.$\\
Following Atiyah, see \cite{At}, we consider the following $\mathbb{H}$-valued 1-form  on $\mathbb{H}$
\begin{equation}
A=(1+\bar{x}x)^{-1}\bar{x}dx=\frac{\bar{x}dx}{1+\bar{x}x}
\end{equation}
which is considered to be  the connection potential of a connection on $\mathbb{H}$. It is easy to see that the curvature 2-form of this connection, $F=dA+A\wedge A$, is given by
$$F=(1+\bar{x}x)^{-1}d\bar{x}\wedge (1+x\bar{x})^{-1}dx=\frac{d\bar{x}\wedge dx}{(1+x\bar{x})^2}$$
which is SD by lemma \ref{AL} part (d). Hence this gives an $\mathbb{H}^*$-SDYM solution on $\mathbb{R}^4$.
This solution is known as the basic instanton. This solution is in fact an $Sp(1)$-SDYM solution. To see this, we identify $Sp(1)$ with quaternions of norm 1, i.e. quaternions $x$ such that $||x||^2=x\bar{x}=1$. Then its Lie algebra, $sp(1)$, is identified with the purely imaginary quaternions, i.e. quaternions with $x_1=0$.  Let $g(x)=(1+\bar{x}x)^{\frac{-1}{2}}$. Under this gauge transformation, $A$ is transformed to  $g^{-1}Ag+g^{-1}dg$.
It is easy to see that 
$$ g^{-1}Ag+g^{-1}dg=(1+\bar{x}x)^{-\frac{1}{2}}\bar{x}dx(1+\bar{x}x)^{-\frac{1}{2}}+(1+\bar{x}x)^{\frac{1}{2}}d(1+\bar{x}x)^{\frac{-1}{2}}$$
$$=Im\{ \frac{\bar{x}dx}{1+x\bar{x}}\}$$
where $Im(x)=\frac{x-\bar{x}}{2}=(0,x_2,x_3,x_4)$. Hence $g^{-1}Ag+g^{-1}dg$ is an $Sp(1)$-connection. This means that this solution is reducible to an $Sp(1)$-solution. \\
We also recall that we have a topological invariant for $Sp(1)$-SDYM solutions in the Euclidean signature called the topological charge of the solution. It is defined to be
(identifying $Sp(1)$ and $SU(2)$)
$$k=\frac{-1}{8\pi^2}\int_{\mathbb{R}^4}tr( F\wedge F)$$
The charge of the basic instanton is
\begin{flalign*}
k & =\frac{-1}{8\pi^2}\int_{\mathbb{R}^4}\frac{-12}{(1+x_1^2+x_2^2+x_3^2+x_4^2)^4}dV\\
& =-\frac{6}{\pi^2}\int_{\mathbb{R}^4}\frac{dx_1\wedge dx_2\wedge dx_3\wedge dx_4}{(1+x_1^2+x_2^2+x_3^2+x_4^2)^4}\\
& =-\frac{6}{\pi^2}\int_{0}^{\infty}\int_{S^3}{\frac{1}{(1+r^4)^2}r^3}d\sigma_{S^3}dr\\
& =-12\int_{0}^{\infty}{\frac{r^3}{(1+r^4)^2}}dr=-1
\end{flalign*}
 It is well-known that this solution extends to an $Sp(1)$-ASDYM solution on $\mathbb{H}\mathbb{P}^1$. Finally we recall that the basic anti-instanton is defined by  \mbox{$(1+x\bar{x})^{-1}xd\bar{x}$} and its topological charge is $1$.

     \end{subsection}

                                                       \begin{subsection}{Basic split instanton}
In this section, using the algebra of split quaternions, we construct an $O(2)$-(A)SDYM solution on $\mathbb{H}_{\mathbb{R}}$. This construction is quite analogous to the construction of the basic (anti-)instanton as explained above.\\
Consider the following connection potential on $\mathbb{H}_{\mathbb{R}}$
 $$A=(1+XX^t)^{-1}XdX^t$$
 where $X=\begin{pmatrix}
 x_{11}&x_{12}\\
 x_{21}&x_{22}
 \end{pmatrix}\in\mathbb{H}_{\mathbb{R}}$. 
 This defines an $\mathbb{H}_{\mathbb{R}}^*$-connection on $\mathbb{H}_{\mathbb{R}}$.  Here 
 $\mathbb{H}_{\mathbb{R}}^*$ is the group of invertible split quaternions. Hence $\mathbb{H}_{\mathbb{R}}^*$ is just $GL(2,\mathbb{R})$.  One can compute the curvature, $F=dA+A\wedge A$, of this connection as follows
 $$d[(1+XX^t)^{-1}X]\wedge dX^t+(1+XX^t)^{-1}XdX^t\wedge (1+XX^t)^{-1}XdX^t=$$
$$-(1+XX^t)^{-1}d(XX^t)(1+XX^t)^{-1}X\wedge dX^t+ 
  (1+XX^t)^{-1}dX\wedge dX^t+$$
  $$(1+XX^t)^{-1}XdX^t\wedge (1+XX^t)^{-1}XdX^t=$$
$$-(1+XX^t)^{-1}dX\wedge X^t(1+XX^t)^{-1}X dX^t+(1+XX^t)^{-1}dX\wedge dX^t=$$
$$(1+XX^t)^{-1}dX\wedge (1-X^t(1+XX^t)^{-1}X)dX^t
$$
which is simply
$$F=(1+XX^t)^{-1}dX\wedge (1+X^tX)^{-1}dX^t$$
This, by lemma \ref{AL} part (a), is SD. We show that this solution is reducible to an $O(2)$-solution. To see this, note that we have 
$$O(2)=\{X\in \mathbb{H}_{\mathbb{R}}| XX^t=X^tX=1\}$$ 
Then its Lie algebra is
$$o(2)=\{X\in \mathbb{H}_{\mathbb{R}}| X+X^t=0\}$$
 Let $g=(1+XX^t)^{\frac{-1}{2}}$. Under this gauge transformation, $A$ is transformed to 
 $g^{-1}Ag+g^{-1}dg$ which is just
 $$(1+XX^t)^{-\frac{1}{2}}XdX^t(1+XX^t)^{-\frac{1}{2}}+(1+XX^t)^{\frac{1}{2}}d[(1+XX^t)^{-\frac{1}{2}}]$$
 It is easy to see that  this defines an $O(2)$-connection. Moreover the curvature 2-form in this new gauge is  
  $$g^{-1}Fg=(1+XX^t)^{-\frac{1}{2}}dX\wedge (1+X^tX)^{-1}dX^t(1+XX^t)^{-\frac{1}{2}}$$
which is equal to 
 $$ \frac{dX\wedge (1+X^tX)^{-1}dX^t }{\sqrt{\det(1+XX^t)}}$$
 Following the Euclidean picture we define the topological charge of the basic split instanton to be 
 \begin{equation}\label{TC}
 \frac{-1}{8\pi^2}\int_{ \mathbb{H}_{\mathbb{R}}}tr(F\wedge F) 
 \end{equation}
Then we have 
 \begin{proposition}
 The topological charge  of the basic split instanton is $1$.
 \end{proposition}                        
  \begin{proof}
We have
$$k=\frac{-1}{8\pi^2}\int_{ \mathbb{H}_{\mathbb{R}}}tr(F\wedge F)=\frac{1}{4\pi^2}\int_{ \mathbb{H}_{\mathbb{R}}}\det(F)=\frac{1}{2\pi^2}\int_{ \mathbb{H}_{\mathbb{R}}}\frac{dV}{\det(1+XX')^2}$$
Using the change of variables 
$$(x_1,x_2,x_3,x_4)=(\frac{x_{11}+x_{22}}{2},\frac{x_{21}-x_{12}}{2},\frac{x_{12}+x_{21}}{2},\frac{x_{22}-x_{11}}{2})$$ 
we have
$$k=\frac{2}{\pi^2}\int_{ \mathbb{R}^4}\frac{dx_1\wedge dx_2\wedge dx_3\wedge dx_4}{(1+x_1^2+x_2^2+x_3^2+x_4^2+(x_1^2+x_2^2-x_3^2-x_4^2)^2)^2}$$
Using polar coordinates on $(x_1,x_2)$ and $(x_3,x_4)$ we can rewrite the integral as
$$k=8\int_{0}^{\infty}\int_{0}^{\infty}{\frac{rs}{(1+2(r^2+s^2)+(r^2-s^2)^2)^2}dr ds}$$
Using change of variables $x=r^2$, $y=s^2$ we have
$$k=2\int_{0}^{\infty}\int_{0}^{\infty}{\frac{1}{(1+2(x+y)+(x-y)^2)^2}dx dy}$$
The change of variables $z=x+y$ and $w=x-y$ gives
$$k=\int_{|w|\leq z}\frac{dz dw}{(1+2z+w^2)^2}=\int_{0}^{\infty}\int_{-z}^{z}\frac{dw}{(1+2z+w^2)^2}dz$$
which is
$$k=\int_{0}^{\infty}\frac{z}{(1+2z)(z+1)^2}+\frac{\arctan(\frac{z}{\sqrt{1+2z}})}{(1+2z)^{\frac{3}{2}}}dz$$
and finally
$$k=\frac{-1}{z+1}-\frac{\arctan(\frac{z}{\sqrt{1+2z}})}{\sqrt{1+2z}}\large{]}_{0}^{\infty}=1$$
  \end{proof}
As we will see, the basic split instanton extends to an $O(2)$-SDYM solution on $\mathbb{H}_{\mathbb{R}}\mathbb{P}^1$.\\ 
It is easy to show that $(1+X^tX)^{-1}X^tdX$ satisfies the ASDYM equations which we call the basic split anti-instanton. Its curvature is given by
$$(1+X^tX)^{-1}dX^t\wedge (1+XX^t)^{-1}dX$$ 
This solution is an $O(2)$-ASDYM solution with topological charge $-1$ and it extends to $\mathbb{H}_{\mathbb{R}}\mathbb{P}^1$.

\end{subsection}                               
                               
  \begin{subsection}{Geometrical constructions and a unifying picture}
As one can see, the basic Euclidean instanton and basic split instanton are very much similar. However,  it can be seen that the analytic continuations of these solutions are singular on the other real form. Nevertheless there is a close relation between these two solutions that we describe in this section. \\
Consider the following smooth $\mathbb{H}_{\mathbb{C}}$-valued $1$-form on $\mathbb{H}_{\mathbb{C}}$
 $$A=(1+Z^*Z)^{-1}Z^*dZ$$
This defines a smooth connection on $\mathbb{H}_{\mathbb{C}}$. Similar to the split case one can see that the curvature of this connection, $F=dA+A\wedge A$, is 
$$F=(1+Z^*Z)^{-1}dZ^*\wedge (1+ZZ^*)^{-1}dZ$$
Note that the curvature is a 2-form  of type $(1,1)$ and hence defines a holomorphic structure on  its associated 2-vector bundle.   Moreover one can see that this connection is reducible to a $U(2)$-connection (similar to the real cases, see previous sections). Now, thanks to lemma \ref{AL1}, we have the following interesting proposition 
\begin{proposition}\label{UC}
The restriction of the connection  $$A=(1+Z^*Z)^{-1}Z^*dZ$$ to $\mathbb{H}$ gives the basic Euclidean instanton and its restriction to $\mathbb{H}_{\mathbb{R}}$ gives the basic split anti-instanton.
\end{proposition}  
Therefore the above connection unifies the basic solutions in the different real forms. Next we explain this relation geometrically. As we saw, $\mathbb{H}_{\mathbb{C}}\mathbb{P}^1$ is isomorphic to $Gr(2,\mathbb{C}^4)$ as a complex manifold. Therefore we can consider the universal vector bundle on $\mathbb{H}_{\mathbb{C}}\mathbb{P}^1$ which we denote by $\mathcal{P}$.  More precisely, the fiber of $\mathcal{P}$ at $[Z:W]$ is the 2-plane generated by 
 $$\begin{pmatrix}
z_{11}\\
z_{21}\\
w_{11}\\
w_{21}
\end{pmatrix}\; \text{and}\; \begin{pmatrix}
z_{12}\\
z_{22}\\
w_{12}\\
w_{22}
\end{pmatrix}$$
where 
 $$Z=\begin{pmatrix}
z_{11} & z_{12}\\
z_{21} & z_{22}
\end{pmatrix}\;\text{and}\;W=\begin{pmatrix}
w_{11} & w_{12}\\
w_{21} & w_{22}
\end{pmatrix}
$$
So $\mathcal{P}$ is a vector sub-bundle of the trivial bundle with fibers $\mathbb{C}^4$.  Let $(\quad,\quad)$ be the standard Hermitian product on $\mathbb{C}^4$, i.e.
 $$( \begin{pmatrix}
 a_1\\
 a_2\\
  a_3\\
 a_4\\
  \end{pmatrix},\begin{pmatrix}
 b_1\\
 b_2\\
  b_3\\
 b_4\\
  \end{pmatrix})=a_1\bar{b}_1+a_2\bar{b}_2+a_3\bar{b}_3+a_4\bar{b}_4
  $$ 
We consider the trivial (smooth) connection $D$ on   $\mathbb{H}_{\mathbb{C}}\mathbb{P}^1\times \mathbb{C}^4$.  The projection of $D$ onto $\mathcal{P}$ gives us a connection $\nabla$ on $\mathcal{P}$. Clearly $\nabla$ is compatible with the Hermitian structure on $\mathcal{P}$ hence $\nabla$ is an $U(2)$-connection.  Now we have
\begin{proposition}\label{UP}
The restriction of $(\mathcal{P},\nabla)$ to $\mathbb{H}$ is the basic Euclidean instanton and its restriction to   $\mathbb{H}_{\mathbb{R}}\mathbb{P}^1$ is the basic split anti-instanton.
\end{proposition} 
\begin{proof}
On $\{[Z:1]|Z\in \mathbb{H}_{\mathbb{C}}\}$, we have the following frame for $\mathcal{P}$
$$u:= \begin{pmatrix}
 z_{11}& z_{12}\\
 z_{21}& z_{22}\\
  1& 0\\
 0& 1\\
  \end{pmatrix}$$
One can see that  the connection potential of $\nabla$ in this local frame $u$ is given by  (see \cite{At})
$$A=(u^{*}u)^{-1}u^*du$$ 
where 
$$u^*=\begin{pmatrix}
 \bar{z}_{11}& \bar{z}_{21}&1&0\\
 \bar{z}_{12}& \bar{z}_{22}&0&1\\
  \end{pmatrix}$$
This implies that 
 $$A=(1+Z^*Z)^{-1}Z^*dZ.$$
 Now proposition \ref{UC} finishes the proof. Note  that the structure group of $\mathcal{P}$ reduces to $Sp(1)$ on $\mathbb{H}$ and to $O(2)$ on $\mathbb{H}_{\mathbb{R}}$.
\end{proof}

  \end{subsection}

 \end{section}
 
 \begin{section}{Moduli space of split anti-instantons of topological charge $-1$}                     
It is well-known that the moduli space of (Euclidean) instantons (i.e. $Sp(1)$-SDYM solutions on $\mathbb{H}\mathbb{P}^1\cong S^4$) of topological charge  $-1$ is isomorphic to $SL(2,\mathbb{H})/Sp(2)$, see \cite{At} for example.                    
The proper conformal group of $\mathbb{H}\mathbb{P}^1$ is $SL(2,\mathbb{H})/\{\pm 1\}$. Hence $SL(2,\mathbb{H})$ acting on the basic  instanton produces solutions of topological charge $-1$. It is well-known that the subgroup of $SL(2,\mathbb{H})$ which fixes the basic instanton (up to gauge transformation) is $Sp(2)$, i.e. the maximal compact subgroup of $SL(2,\mathbb{H})$. It is well-known that all $Sp(1)$-SDYM solutions of topological charge $1$ on $\mathbb{H}\mathbb{P}^1$ can be obtained this way and hence we have that the moduli space of $Sp(1)$-SDYM solutions of topological charge $-1$ on $\mathbb{H}\mathbb{P}^1$ is isomorphic to $SL(2,\mathbb{H})/Sp(2)$.\\
In this section we want to consider the moduli problem in the split case.  As we will see the moduli space of $O(2)$-ASDYM solutions of topological charge $-1$ on $\mathbb{H}_{\mathbb{R}}\mathbb{P}^1$ is infinite dimensional. Nevertheless, we will see that by imposing a condition on the solutions we obtain a finite dimensional space of solutions which is isomorphic to    $SL(2,\mathbb{H}_\mathbb{R})/SO(4)\cong SL(4,\mathbb{R})/SO(4)$. This is in complete parallel with the Euclidean case because  the proper conformal group of $\mathbb{H}_{\mathbb{R}}\mathbb{P}^1$ is $SL(4,\mathbb{R})/\{\pm 1\}$ and $SO(4)$ is the maximal compact subgroup of $SL(4,\mathbb{R})$.  

 \begin{subsection}{Preliminaries}
 First we briefly explain the classification of $O(2)$ principal bundles (or equivalently real orthogonal vector bundles of rank 2) on 
 $\mathbb{H}_{\mathbb{R}}\mathbb{P}^1$.  \\
 The split quaternionic projective line has a double cover denoted by  $ \widetilde{\mathbb{H}_{\mathbb{R}}\mathbb{P}^1}$ 
 which is isomorphic to the Grassmannian of oriented 2-planes in $\mathbb{R}^4$.  More precisely on 
 $$S:=\{(Z,W)\in \mathbb{H}_{\mathbb{C}}^2|\; \nexists A\in \mathbb{H}_{\mathbb{C}}\setminus \{0\} \quad  \text{s.t.} \quad ZA=WA=0\}$$
 we define a weaker equivalence relation. Two elements $(X,Y), (X_1,Y_1)\in S$ are called equivalent if there is  an  element $q\in  \mathbb{H}_{\mathbb{R}}$ with positive determinant such that $X_1=Xq, Y_1=Yq$. Let $ \widetilde{\mathbb{H}_{\mathbb{R}}\mathbb{P}^1}$, called the oriented split quaternionic projective line, be the set of equivalence classes of elements in $S$.  
The equivalence class of $(X,Y)\in S$ in $ \widetilde{\mathbb{H}_{\mathbb{R}}\mathbb{P}^1}$
is denoted by $\{X:Y\}$.  It is well-known that $ \widetilde{\mathbb{H}_{\mathbb{R}}\mathbb{P}^1}$ is a smooth manifold diffeomorphic to $S^2\times S^2$. We have an obvious map $\pi:\widetilde{\mathbb{H}_{\mathbb{R}}\mathbb{P}^1}\to \mathbb{H}_{\mathbb{R}}\mathbb{P}^1$. Moreover we have the following isomorphism
$\sigma: \widetilde{\mathbb{H}_{\mathbb{R}}\mathbb{P}^1}\to \widetilde{\mathbb{H}_{\mathbb{R}}\mathbb{P}^1}$
given by
$$\sigma(\{X:Y\}]=\{XJ:YJ\}$$
where $J:=\begin{pmatrix}
0 &1 \\
1 & 0
\end{pmatrix}
$.  Clearly we have $\pi\circ\sigma=\pi$.\\
 It is well-known that  real line bundles on any manifold $M$ are  classified by their first Stiefel-Whitney class $w_1\in H^1(M,\mathbb{Z}_2)$. Likewise complex line bundles on $M$ are  classified by their first Chern class $c_1\in H^2(M,\mathbb{Z})$.\\ 
For $O(2)$-principal  bundles or equivalently real orthogonal vector bundles of rank 2, the first  invariant is the Stiefel-Whitney class $w_1\in H^1(M,\mathbb{Z}_2)$. 
This invariant is zero iff the vector bundle is orientable or equivalently the  principal $O(2)$-bundle is in fact a  principal $SO(2)$-bundle.  Therefore if $H^1(M,\mathbb{Z}_2)=0$, every principal $O(2)$-bundle is induced by a principal $SO(2)$-bundle. Set $\bar{H}^2(M,\mathbb{Z})=H^2(M,\mathbb{Z})/\sim$ where $a\sim b$ if and only if $a=b$ or $a=-b$ (note that  $\bar{H}^2(M,\mathbb{Z})$ is not a group).
\begin{proposition}
Suppose that $H^1(M,\mathbb{Z}_2)=0$. Then there is a 1-1 correspondence between the isomorphism classes of principal $O(2)$-bundles on $M$ and $\bar{H}^2(M,\mathbb{Z})$. 
\end{proposition}
\begin{proof}
Since  $H^1(M,\mathbb{Z}_2)=0$, the isomorphism classes of principal $SO(2)$-bundles  on $M$ are in 1-1 correspondence between pairs $(E,o)$ where $E$ is a real orthogonal vector bundle on $M$ of rank two and $o$ is an orientation on $E$. Principal $SO(2)$-bundles $P$ on $M$ are classified by $c_1(P)\in H^2(M,\mathbb{Z})$. Moreover $c_1(E,-o)=-c_1(E,o)$ where $-o$ is the opposite orientation to $o$. Hence principal $O(2)$-bundles on $M$ are classified by 
$$c_1(E):=\overline{c_1(E,o)}\in \bar{H}^2(M,\mathbb{Z})$$
 where $o$ is an orientation on $E$. 
\end{proof}
For a real orthogonal vector bundle $E$ of rank two on $M$ with \mbox{$H^1(M,\mathbb{Z}_2)=0$,} we call 
$c_1(E):=\overline{c_1(E,o)}$ the first Chern class of $E$. \\
The cohomology groups of  $ \mathbb{H}_{\mathbb{R}}\mathbb{P}^1$ have been computed by Ehresmann, see 
\cite{Eh} or \cite{Ch}, as follows 
$$H^1(  \mathbb{H}_{\mathbb{R}}\mathbb{P}^1,\mathbb{Z}_2)\cong H^2(  \mathbb{H}_{\mathbb{R}}\mathbb{P}^1,\mathbb{Z})\cong \mathbb{Z}_2$$
Therefore, up to isomorphism, there are only two real line bundles on $\mathbb{H}_{\mathbb{R}}\mathbb{P}^1$, namely the trivial one $\varepsilon$ and the nontrivial one which we denote by $\widetilde{\varepsilon}$.
We always realize $\widetilde{\varepsilon}$ as the line bundle on $\mathbb{H}_{\mathbb{R}}\mathbb{P}^1$ whose sheaf of sections is the real-valued smooth functions on $\widetilde{\mathbb{H}_{\mathbb{R}}\mathbb{P}^1}$ such that $f(\sigma(x))=-f(x)$, i.e. odd functions. 
Similarly, up to isomorphism, there are only two complex line bundles on $\mathbb{H}_{\mathbb{R}}\mathbb{P}^1$, namely  $\varepsilon_{\mathbb{C}}$ and the nontrivial one which we denote by $\mathcal{L}$. In fact, under the isomorphism $H^1(\mathbb{H}_{\mathbb{R}}\mathbb{P}^1,\mathbb{Z}_2)\to H^2(\mathbb{H}_{\mathbb{R}}\mathbb{P}^1,\mathbb{Z})$ coming from the exact sequence $0\to\mathbb{Z}\to\mathbb{Z}\to\mathbb{Z}_2\to 0$, we have $c_1(\mathcal{L})=
w_1(\widetilde{\varepsilon})$ and hence 
$$\mathcal{L}=\widetilde{\varepsilon}_{\mathbb{C}}=\widetilde{\varepsilon}\otimes \mathbb{C}$$
 Up to isomorphism, there are only two orientable real orthogonal vector bundle of rank two on $\mathbb{H}_{\mathbb{R}}\mathbb{P}^1$  namely the trivial one and $\widetilde{\varepsilon}\oplus \widetilde{\varepsilon}$. 
Now suppose that $E$ is a real orthogonal vector bundle of rank 2 on $ \mathbb{H}_{\mathbb{R}}\mathbb{P}^1$. Then $\pi^*E$ is  a  real orthogonal vector bundle of rank 2 on $\widetilde{ \mathbb{H}_{\mathbb{R}}\mathbb{P}^1}$. Since $H^1( \widetilde{ \mathbb{H}_{\mathbb{R}}\mathbb{P}^1},\mathbb{Z}_2)=0$, we can consider $c_1(\pi^*E)$.  Using the isomorphism 
$$H^2(\mathbb{H}_{\mathbb{R}}\mathbb{P}^1,\widetilde{\mathbb{Z}})\to H^2(\widetilde{ \mathbb{H}_{\mathbb{R}}\mathbb{P}^1},\mathbb{Z})$$
we can assign an element $\widetilde{c_1(E)}$
in $\bar{H}^2(\mathbb{H}_{\mathbb{R}}\mathbb{P}^1,\widetilde{\mathbb{Z}})$ to $E$ which corresponds to  $c_1(\pi^*E)$. We call it the first twisted Chern class of $E$. In order to state the classification result we need the following definition
\begin{definition}
Real vector bundles $E$ and $F$ on $\mathbb{H}_{\mathbb{R}}\mathbb{P}^1$  are called T-isomorphic
if $E$ is isomorphic to $F$ or $\widetilde{F}:=F\otimes \widetilde{\varepsilon}$. 
\end{definition}
For more on twisting sheaves see \cite{Ar}.
\begin{proposition}
The T-isomorphism classes of real orthogonal vector bundles of rank two on $\mathbb{H}_{\mathbb{R}}\mathbb{P}^1$ are in 1-1 correspondence with  $\bar{H}^2(\mathbb{H}_{\mathbb{R}}\mathbb{P}^1,\widetilde{\mathbb{Z}})$. The correspondence is given by $E\longmapsto \widetilde{c_1(E)}$. 
\end{proposition}
\begin{proof}
Since $\mathbb{H}_{\mathbb{R}}\mathbb{P}^1\cong \widetilde{\mathbb{H}_{\mathbb{R}}\mathbb{P}^1}/\mathbb{Z}_2$,
it is well-known that the isomorphism classes of real orthogonal vector bundles of rank two on  $\mathbb{H}_{\mathbb{R}}\mathbb{P}^1$ are in 1-1 correspondence with pairs $(F,\alpha)$ where $F$ is a real orthogonal vector bundle on $ \widetilde{\mathbb{H}_{\mathbb{R}}\mathbb{P}^1}$ and $\alpha:F\to\sigma^*F$ is an isomorphism such that $\alpha(\sigma(x))\alpha(x)=1$ for any $x$, see \cite{At1}. One can see that for each real orthogonal vector bundle of rank 2 on $ \widetilde{\mathbb{H}_{\mathbb{R}}\mathbb{P}^1}$
there are exactly two maps $\alpha_1$ and $\alpha_2$ with this property up to isomorphism. Moreover 
if $(F,\alpha_1)$ induces $E$ on $\mathbb{H}_{\mathbb{R}}\mathbb{P}^1$, then $(F,\alpha_2)$ induces 
$\widetilde{E}$. Hence $E\longmapsto \widetilde{c_1(E)}$ gives a 1-1 correspondence. 
\end{proof}
We have 
 $$H^2(\mathbb{H}_{\mathbb{R}}\mathbb{P}^1,\widetilde{\mathbb{Z}})=\mathbb{Z}\oplus\mathbb{Z}$$
so the T-isomorphic classes of real orthogonal vector bundles of rank two on   $\mathbb{H}_{\mathbb{R}}\mathbb{P}^1$ are classified by $\widetilde{c_1(E)}\in\mathbb{Z}\oplus\mathbb{Z}/\sim$. \\
From now on suppose that $V$ is a real orthogonal vector bundle of rank 2 on $\mathbb{H}_{\mathbb{R}}\mathbb{P}^1$.  As far as we are concerned with SDYM equations, there is no difference between these equations on $V$ and the ones on $\widetilde{V}$. Hence we just consider $V$ up to T-isomorphism. Let $g_V$ be the sub-bundle of $End(V)$ consisting of elements which are anti-symmetric with respect to the orthogonal structure of $V$. Therefore $g_V$ is just a real line bundle. It is easy to see that $g_V$ is trivial if and only if $V$ is trivial (up to T-isomorphism). We assume that $V$ is not trivial and hence $g_V$ is not trivial and hence isomorphic to $\widetilde{\varepsilon}$.   Choosing an orientation $o$ on $\pi^*V$ gives a canonical isomorphism $f:g_{(\pi^*V,o)}\to \varepsilon_{\widetilde{\mathbb{H}_{\mathbb{R}}\mathbb{P}^1}}$. If we change the orientation, this canonical isomorphism changes to
$-f$.  These isomorphisms descend to $\mathbb{H}_{\mathbb{R}}\mathbb{P}^1$. Therefore, if $V$ is not orientable, we have an isomorphism $g_V\to\widetilde{\varepsilon}$ which is canonical up to a negative sign. Under this isomorphism we have the twisted de Rham complex on $\mathbb{H}_{\mathbb{R}}\mathbb{P}^1$ 
\begin{equation}\label{TES}
0\to \Lambda^0(g_V)\overset{d}{\to}\Lambda^1(g_V)\overset{d}{\to}\Lambda^2(g_V)\overset{d}{\to}\Lambda^3(g_V)\overset{d}{\to}\Lambda^4(g_V)\to 0
\end{equation}
where $\Lambda^i(g_V)$ is the space of smooth $g_V$-valued $i$-forms on  $\mathbb{H}_{\mathbb{R}}\mathbb{P}^1$. 
 To $V$ we can associate two invariants namely its first twisted Chern class $\widetilde{c_1(V)}\in \widetilde{H}^2(\mathbb{H}_{\mathbb{R}}\mathbb{P}^1,\widetilde{\mathbb{Z}})$ and   its first  Pontryagin class  $p_1(V)\in H^4(\mathbb{H}_{\mathbb{R}}\mathbb{P}^1,\mathbb{Z})$.  Then we have the following proposition relating these two invariants
\begin{proposition}
(a) Suppose that $\nabla$ is an $O(2)$-connection on $E$ and $C\in \Lambda^2(g_V)$ is its curvature.  Then 
$C$ is closed as an element in $\Lambda^2(g_V)$. Moreover $[C]\in \bar{H}^2_{DRT} $ in the second cohomology group of the twisted exact sequence (sequence \ref{TES}) 
is a well-defined element which only depends on $V$ and $[\frac{C}{2\pi}]=\widetilde{c_1(E)}$ under the natural isomorphism 
$$\bar{H}^2(\mathbb{H}_{\mathbb{R}}\mathbb{P}^1,\widetilde{\mathbb{R}})\to
\bar{H}^2_{DRT} $$
(b) The cup product $\widetilde{c_1(V)}\cup \widetilde{c_1(V)}\in H^4(\mathbb{H}_{\mathbb{R}}\mathbb{P}^1,\mathbb{Z})$ is well defined and we have 
$p_1(V)=-\widetilde{c_1(V)}\cup \widetilde{c_1(V)}$.
\end{proposition}
\begin{proof}
(a) It is enough to prove the corresponding statement for a real orthogonal vector bundle $V$ of rank 2 on 
 $ \widetilde{\mathbb{H}_{\mathbb{R}}\mathbb{P}^1}$.  Let $\nabla$ be an $O(2)$-connection on $V$. We fix an orientation on $V$. Then $\nabla$ is also an $SO(2)$-connection. We can consider a complex structure on $V$ compatible with the orientation and the orthogonal structure. Then $V$ becomes a complex line bundle and the curvature of $\nabla$ as an $U(1)$-connection is a purely imaginary 2-form $\omega$ under the canonical isomorphism $g_V\to i\mathbb{R}$. We also have the canonical isomorphism $g_V\to\mathbb{R}$ as discussed  above. These two isomorphisms are the same up to multiplication by $i$. The proof is complete by noting that $[\frac{i}{2\pi}\omega]=c_1(V)$.  \\
 (b) If we define $\widetilde{c_1(V)}\cup \widetilde{c_1(V)}$    using isomorphism $f$ or $-f$ we get the same answer so we have a well-defined cohomology class       
$\widetilde{c_1(V)}\cup \widetilde{c_1(V)}\in H^4(\mathbb{H}_{\mathbb{R}}\mathbb{P}^1,\mathbb{Z})$.  It is easy to see that $\widetilde{c_1(V)}\cup \widetilde{c_1(V)}=[\frac{-tr(C\wedge C)}{8\pi^2}]$. On the other hand 
$$p_1(E)=[\frac{-\det(C)}{4\pi^2}] $$
 We claim that $tr(C\wedge C)=-2 \det(C)$. In fact in an orthogonal frame the curvature is given by $\begin{pmatrix}
 0& -\alpha\\
 \alpha& 0
 \end{pmatrix}$
 where $\alpha$ is just a 2-form. Then it is easy to see that $tr(C\wedge C)=-2 \det(C)=-2\alpha\wedge\alpha$.  Therefore we have 
 $p_1(V)=-\widetilde{c_1(V)}\cup \widetilde{c_1(V)}$
 
 \end{proof}
Moreover we have the following observation relating the topological charge and the first Pontryagin class.
\begin{proposition}
The topological charge of an $O(2)$-SDYM solution on $\mathbb{H}_{\mathbb{R}}\mathbb{P}^1$ is just   the negative of the first Pontryagin class of 
its vector bundle evaluated on the fundamental class of $\mathbb{H}_{\mathbb{R}}\mathbb{P}^1$.
\end{proposition}
  
 \end{subsection}                   
    
  \begin{subsection}{$O(2)$-ASDYM Solutions of topological charge $-1$}
 Suppose that $V$ is a real orthogonal vector bundle of rank 2 on $\mathbb{H}_{\mathbb{R}}\mathbb{P}^1$. We want to describe the moduli space of $O(2)$-solutions on $V$. 
 Let $G_V$ be the set of all gauge transformations of $V$, i.e. the group of all bundle isomorphisms $f:V\to V$ which preserve the orthogonal structure. Let $\mathcal{A}_V$ be the space of $O(2)$-connections on $V$. It is well-known that $\mathcal{A}_V$ is an affine space on $\Lambda^1(g_V)$. The group of gauge transformations acts on $\mathcal{A}_V$ via $g.\nabla:=g\nabla g^{-1}$. We assume that $V$ admits an $O(2)$-ASDYM solution and let $\mathcal{B}_V\subset \mathcal{A}_V$ be the space of all $O(2)$-solutions on $V$. 
It is easy to see that $G_V$ preserves $\mathcal{B}_V$. By the moduli space of $O(2)$-solutions on $V$  we mean $\mathcal{M}_V:=\mathcal{B}_V/G_V$. Since the structure group, namely $O(2)$, is not connected, $G_V$ is not also connected and we have $G_V=G^{+}_V\cup G^{-}_V$ where $G^{\pm}_V$ is the set of gauge transformations with determinant $\pm$. We set  $\mathcal{M}^+_V:=\mathcal{B}_V/G^+_V$. We would like to describe $\mathcal{M}^+_V:=\mathcal{B}_V/G^+_V$.  We note that $G^{+}_V$ is a commutative group. Let $\mathcal{G}^{+}_V$ be the sheaf of gauge transformations with determinant 1. Then we have an exponential map $\exp:g_V\to\mathcal{G}^+_V$ coming form the pointwise exponential map.   It is easy to see that the following sequence of sheaves 
$$0\to\widetilde{\mathbb{R}}\to g_V\overset{\exp}{\to}\mathcal{G}^{+}_V\to0$$
is exact.  The corresponding exact sequence of global sections and the fact that $H^1(\mathbb{H}_{\mathbb{R}}\mathbb{P}^1,\widetilde{\mathbb{R}})=0$ implies that 
$\exp:\Lambda^0(g_V)\to G^+_V$ is onto.  Now suppose that $g\in G^+_V$ and $\nabla\in \mathcal{A}_V$. Then $g=\exp(-h)$ for some $h\in\Lambda^0(g_V)$ and hence we have 
$$g.\nabla=\nabla+gdg^{-1}=\nabla+dh$$
in other words two $O(2)$-connections on $V$ are gauge equivalent (under the action of  $G^+_V$) if and only if their difference is an \emph{exact} $g_V$-valued one form. This implies that $\mathcal{A}_V/G^+_V$ is isomorphic to $\Lambda^1(g_V)/d\Lambda^0(g_V)$. Moreover if we fix  $\nabla_0\in \mathcal{B}_V$ then $\nabla_0+\alpha$ is ASD if and only if $\alpha\in \Lambda^1(g_V)$ is ASD. Therefore 
 $\mathcal{M}^+_V$ is isomorphic to $\Lambda^1_{ASD}(g_V)/d\Lambda^0(g_V)$ where $\Lambda^1_{ASD}(g_V)$ is the space of $g_V$-valued 1-forms $\alpha$ for which $d\alpha$ is ASD.  Therefore it is important to consider $\Lambda^2_{ASD}(g_V)$ the space of closed ASD $g_V$-valued 2-forms. In other words we are led to consider  (twisted) Maxwell's equations on $\mathbb{H}_{\mathbb{R}}\mathbb{P}^1$. A description of the solutions  is given by Guilleman and Sternberg \cite{GS}
 \begin{theorem}\label{GS}
 There is an $SL(4,\mathbb{R})$-equivariant  transformation 
 $$R:\Lambda^3_{\mathbb{R}\mathbb{P}^3}\to \Lambda^2_{ASD}(g_V)$$
 which is a bijection onto the space of closed ASD $g_V$-valued 2-forms. Moreover the space of exact 3-forms is mapped onto the space of exact ASD 2-forms. Here $\mathbb{R}\mathbb{P}^3$ is the real projective 3-space and $\Lambda^3_{\mathbb{R}\mathbb{P}^3}$ is the space of 3-forms on it. 
  \end{theorem} 
This theorem implies the following
\begin{theorem}
Suppose that $\mathcal{B}_V$ is not empty. Then there is an isomorphism
$$\mathcal{M}^+_V\cong \Lambda^3_{\mathbb{R}\mathbb{P}^3, \text{exact}}$$
where $\Lambda^3_{\mathbb{R}\mathbb{P}^3, \text{exact}}$ is the space of exact 3-forms on $\mathbb{R}\mathbb{P}^3$. In particular the moduli space of $O(2)$-ASDYM solutions on $V$ is infinite dimensional. 
\end{theorem} 
 \begin{proof}
 Since the action of $G^+_V$ on the curvature $C_{\nabla}$ of a connection $\nabla\in\mathcal{A}_V$ is trivial, we obtain a well-defined map
 $$\phi:\mathcal{M}^+_V\to\Lambda^2(g_V)$$
 by sending $\nabla$ to $C_\nabla$.   As we saw,  $C_\nabla$ is independent of $\nabla$ up to addition by an exact ASD $g_V$-valued 2-form  and it only depends on $V$. Using theorem \ref{GS}, we obtain the isomorphism.
 \end{proof}
The above two theorems give a restriction on the twisted Chern class of $V$. More precisely, since  $H^3(\mathbb{R}\mathbb{P}^3,\mathbb{R})\cong \mathbb{R}$ we see that not every twisted Chern class is possible for real orthogonal vector bundles  of rank two which admit an $O(2)$-ASDYM solution. In particular we see that the topological charge of the solution as defined by \ref{TC} determines the vector bundle up to T-isomorphism. Therefore we can restate the last theorem as 
\begin{theorem}\label{MS}
The moduli space of $O(2)$-ASDYM solutions of charge $-1$ on $\mathbb{H}_{\mathbb{R}}\mathbb{P}^1$ is isomorphic to  $\Lambda^3_{\mathbb{R}\mathbb{P}^3, \text{exact}}$.
\end{theorem} 
 \begin{proof}
 Note that the basic split anti-instanton has topological charge $-1$.   
 \end{proof}
 
  \end{subsection}

          \begin{subsection}{Moduli space of split anti-instantons of topological charge $-1$}
 As we saw in the last section, the moduli space of $O(2)$-ASDYM solutions of topological charge $-1$   on $\mathbb{H}_{\mathbb{R}}\mathbb{P}^1$ is infinite dimensional. However the basic split anti-instanton  is a very special $O(2)$-ASDYM solution as we will see.  Let us recall the geometrical construction of the basic split anti-instanton on  $\mathbb{H}_{\mathbb{R}}\mathbb{P}^1$. 
Let $\mathcal{P}$ be the universal vector bundle on  $ \mathbb{H}_{\mathbb{R}}\mathbb{P}^1$. Setting
 $X= \begin{pmatrix}
 x_{11}& x_{12}\\
 x_{21}& x_{22}
 \end{pmatrix}$ and  $Y= \begin{pmatrix}
 y_{11}& y_{12}\\
 y_{21}& y_{22}
 \end{pmatrix}$, then the fiber of  $\mathcal{P}$ at $[X:Y]$ is the plane in $\mathbb{R}^4$ spanned by 
  $ \begin{pmatrix}
 x_{11}& x_{12}\\
 x_{21}& x_{22}\\
  y_{11}& y_{12}\\
 y_{21}& y_{22}\\
  \end{pmatrix}$. Then consider the standard inner product $(\quad,\quad)$ on $\mathbb{R}^4$, i.e.
  $$( \begin{pmatrix}
 a_1\\
 a_2\\
  a_3\\
 a_4\\
  \end{pmatrix},\begin{pmatrix}
 b_1\\
 b_2\\
  b_3\\
 b_4\\
  \end{pmatrix})=a_1b_1+a_2b_2+a_3b_3+a_4b_4.
  $$
 Finally let $\nabla$  be the projection of the trivial connection on $\mathbb{R}^4$ onto $\mathcal{P}$. Then $(\mathcal{P},\nabla)$ is the basic split instanton which is an $O(2)$-ASDYM solution of topological charge $-1$.     
Since the ASDYM equations  are conformally invariant, the pull back of any ASDYM solution under a conformal map of $ \mathbb{H}_{\mathbb{R}}\mathbb{P}^1$ is again an ASDYM solution.  It is known that the proper conformal group of $ \mathbb{H}_{\mathbb{R}}\mathbb{P}^1$  is isomorphic to $SL(4,\mathbb{R})/\{\pm1\}$. More precisely, given any    $g=  \begin{pmatrix}
 a& b\\
 c& d
 \end{pmatrix}\in SL(4,\mathbb{R})$ with $a,b,c,d\in \mathbb{H}_{\mathbb{R}}$ the map
 $$g([X:Y]):=[aX+bY:cX+dY]$$
is a conformal map and all the proper conformal maps of $ \mathbb{H}_{\mathbb{R}}\mathbb{P}^1$  are achieved in this way. Hence $g$ acting on the basic split anti-instanton gives another solution of $O(2)$-ASDYM of topological charge $-1$.   It is easy to see the action of $g$ on $(\mathcal{P},\nabla)$ geometrically. Consider the sub-bundle of the trivial bundle on $ \mathbb{H}_{\mathbb{R}}\mathbb{P}^1$ with fibers $\mathbb{R}^4$ 
whose fiber at $[X:Y]$ is spanned by 
$$g \begin{pmatrix} 
 X\\
 Y
 \end{pmatrix}=\begin{pmatrix}
 a\begin{pmatrix}
 x_{11}& x_{12}\\
 x_{21}& x_{22}\\
 \end{pmatrix}+b\begin{pmatrix}
 y_{11}& y_{12}\\
 y_{21}& y_{22}\\
 \end{pmatrix} \\
  c\begin{pmatrix}
 x_{11}& x_{12}\\
 x_{21}& x_{22}\\
 \end{pmatrix}+d\begin{pmatrix}
 y_{11}& y_{12}\\
 y_{21}& y_{22}\\
 \end{pmatrix} 
    \end{pmatrix}
$$
Then the projection of the trivial connection on $\mathbb{R}^4$ to this vector bundle is the solution obtained from the action of $g$ on the basic split anti-instanton. From this construction it is clear that if $g\in SO(4)$ then the solution is gauge equivalent  to the basic anti-split instanton.  In other words the basic split anti-instanton has a big symmetry group. This is a property which is very restrictive for a solution.  We define anti-instantons of topological charge $-1$ as follows 
 \begin{definition}
 A split anti-instanton of topological charge $-1$ is an $O(2)$-ASDYM solution of topological charge $-1$ on  $ \mathbb{H}_{\mathbb{R}}\mathbb{P}^1$ which is invariant (up to gauge transformation) under some maximal compact subgroup of $SL(4,\mathbb{R})$.
 \end{definition}
The main result of this section is that the moduli space of split anti-instantons of topological charge $-1$  is finite dimensional and it is isomorphic to $SL(4,\mathbb{R})/SO(4)$.  First we show the following
 \begin{proposition}\label{IS}
The subgroup of $SL(4,\mathbb{R})$ which leaves the basic split anti-instanton invariant (up to gauge transformation) is $SO(4)$.
\end{proposition} 
\begin{proof}
Suppose that $g$ fixes the basic split anti-instanton. Using the action of $SO(4)$ we can assume that 
$c=0$ and $a, b$ are upper-triangular. Computations similar to the basic split anti-instanton show that the connection potential of this solution  (in suitable local coordinates and  local frame, see \ref{UP}) is given by 
$$A_g=[d^td+(aX+b)^t(aX+b)]^{-1}(aX+b)^tadX$$
In order to compute the curvature we need the following lemma
\begin{lemma}
For any two elements $a,b\in \mathbb{H}_{\mathbb{R}}$ such that $b$ is invertible we have
$$1-a(a^ta+b^tb)^{-1}a^t=(1+a(b^tb)^{-1}a^t)^{-1}$$
\end{lemma}
\begin{proof}
We simply have
$$(1-a(a^ta+b^tb)^{-1}a^t)(1+a(b^tb)^{-1}a^t)=1+a(b^tb)^{-1}a^t-a(a^ta+b^tb)^{-1}a^t(1+a(b^tb)^{-1}a^t)
$$
$$=1+a(b^tb)^{-1}a^t-a(a^ta+b^tb)^{-1}a^t-a(a^ta+b^tb)^{-1}a^ta(b^tb)^{-1}a^t=$$
$$
1+a(b^tb)^{-1}a^t-a(a^ta+b^tb)^{-1}a^t-a(a^ta+b^tb)^{-1}(a^ta+b^tb)(b^tb)^{-1}a^t+a(a^ta+b^tb)^{-1}a^t
$$
which simplifies to 1. Hence we have the formula.

\end{proof}
From this lemma and computations similar to the ones for the curvature of the basic split anti-instanton, we find the following formula for the curvature $F_g$ of $A_g$
$$F_g= [d^td+(aX+b)^t(aX+b)]^{-1}dX^t\wedge a^t[1+(aX+b)(d^td)^{-1}(aX+b)^t]^{-1}adX$$
 This formula implies that
$$\det(F_g)=\frac{2 \det(a)^2dV}{\det[d^td+(aX+b)^t(aX+b)]\det[1+(aX+b)(d^td)^{-1}(aX+b)^t]}$$
Therefore the field strength $\sqrt{(\det(F_g),\det(F_g))}$ (the induced metric on the the vector bundle of 4-forms is always positive) attains its maximum at the point $X=-a^{-1}b$ (the center) and its value is 
$2\det(a)^2\det(d)^{-2}$ (the scale). For example the center of the basic split anti-instanton is the origin and its scale is 2. \\
Since $\det(F_g)$ is gauge invariant (so are the center and the scale) and $g$ fixes the basic split anti-instanton we must have $b=0$ and can assume that $\det(a)=\det(d)=1$.  With this assumption $\det(F_g)$ becomes
$$\det(F_g)=\frac{2 dV}{\det[d^td+X^ta^taX]\det[1+aX(d^td)^{-1}X^ta^t]}$$
$$=\frac{2 dV}{\det[1+aXd^{-1}(aXd^{-1})^t]^2}$$
So if $g$ fixes the basic split anti-instanton, then we have
$$\det[1+aXd^{-1}(aXd^{-1})^t]=\det[1+XX^t]$$
for any $X\in \mathbb{H}_{\mathbb{R}}$. This is equivalent to the following
$$\det[1+aX(aX)^t]=\det[1+Xd(Xd)^t] (*)$$
for any $X\in \mathbb{H}_{\mathbb{R}}$.  Set
$$a=\begin{pmatrix}
a_{11} & a_{12}\\
0 & a_{11}^{-1}
\end{pmatrix},
d=\begin{pmatrix}
d_{11} & d_{12}\\
0 & d_{11}^{-1}
\end{pmatrix} 
$$
Then (*) implies that
$$(a_{11}x_{11}+a_{12}x_{21})^2+(a_{11}x_{12}+a_{12}x_{22})^2+(a_{11}^{-1}x_{21})^2
+(a_{11}^{-1}x_{21})^2
$$
$$=(d_{11}x_{11})^2+(d_{12}x_{11}+d_{11}^{-1}x_{12})^2+(d_{11}x_{21})^2
+(d_{12}x_{21}+d_{11}^{-1}x_{22})^2
$$
 Comparing the coefficients implies that $a=d=\pm1$. Hence $g\in SO(4)$. 
\end{proof} 
Using this proposition we have
\begin{theorem}
 The moduli space of split anti-instantons of topological charge $-1$  is isomorphic to $SL(4,\mathbb{R})/SO(4)$.
\end{theorem}
 \begin{proof}
 Suppose that $\nabla$ is an $O(2)$-ASDYM solution of topological charge $-1$ which is invariant under a maximal compact subgroup of 
 $SL(4,\mathbb{R})$. Since all maximal compact subgroups of $SL(4,\mathbb{R})$ are conjugate, we may assume that $\nabla$ is invariant under the action of $SO(4)$.  Thanks to  theorem \cite{GS} and the fact that the action of the gauge transformation on the curvature is $\pm$, we deduce that $\omega\in\Lambda^3_{\mathbb{R}\mathbb{P}^3}$   goes to $\pm\omega$ under the action of $SO(4)$ where  $C_\nabla$, the curvature of $\nabla$, is equal to $R(\omega)$, see theorem \cite{GS}. It is easy to see that there is only one nonzero $\omega_0\in\Lambda^3_{\mathbb{R}\mathbb{P}^3}$ (up to multiplication by scalars) with this property and it is in fact invariant under the action of $SO(4)$. This implies that $\nabla$ has to be the basic split anti-instanton up to gauge transformation.  Now proposition \ref{IS} finishes the proof of the theorem.

 \end{proof}

\end{subsection}

                       \end{section}

                                               \begin{section}{t'Hooft ansatz}
t'Hooft ansatz gives a way to produce ASDYM solutions in the Euclidean case by starting from solutions of the Laplacian on four variables, see \cite{MW}.  Even though it does not produce the whole set of solutions in general, it does produce all instantons of topological charge $-1$. In this section we give the analog of t'Hooft ansatz in the split case and show that every split anti-instanton can be obtained via split t'Hooft ansatz.  

                \begin{subsection}{Grassmannian of 2-planes in $\mathbb{R}^4$}     
As we saw, $\mathbb{H}_{\mathbb{R}}\mathbb{P}^1$ is isomorphic to   the Grassmannian of 2-planes in $\mathbb{R}^4$. Since we want to do calculations in local coordinates in this section, we work with  the Grassmannian of 2-planes in $\mathbb{R}^4$.   We denote the Grassmannian of 2-planes in $\mathbb{R}^4$ by $Gr$. We denote elements of $Gr$  by  $$p=  \left | \begin{array}{cccc}
a_{11}&b_{11} \\
a_{21} & b_{21} \\
a_{31} & b_{31}\\
a_{41} & b_{41}
 \end{array} \right |$$
which means 
$$p=\mathbb{R} \left ( \begin{array}{cccc}
a_{11} \\
a_{21} \\
a_{31} \\
a_{41} 
 \end{array} \right )+\mathbb{R}\left ( \begin{array}{cccc}
b_{11} \\
b_{21} \\
b_{31}\\
 b_{41}
 \end{array}\right ) $$
Therefore, 
$$\left | \begin{array}{cccc}
a_{11}&b_{11} \\
a_{21} & b_{21} \\
a_{31} & b_{31}\\
a_{41} & b_{41}
 \end{array} \right |=\left | \begin{array}{cccc}
c_{11}&d_{11} \\
c_{21} & d_{21} \\
c_{31} & d_{31}\\
c_{41} & d_{41}
 \end{array} \right |$$
if and only if
$$ \left ( \begin{array}{cccc}
a_{11}&b_{11} \\
a_{21} & b_{21} \\
a_{31} & b_{31}\\
a_{41} & b_{41}
 \end{array} \right )=\left ( \begin{array}{cccc}
c_{11}&d_{11} \\
c_{21} & d_{21} \\
c_{31} & d_{31}\\
c_{41} & d_{41}
 \end{array} \right ) g$$
for some $g\in GL(2,\mathbb{R})$.
We introduce the following local coordinates on $Gr$. For any $g\in SL(4,\mathbb{R})$, we define  $\Psi_g:M_{2}(\mathbb{R})\to Gr$ by 

$$ \Psi_g(A)=\left | \begin{array}{cc}
g_{11}A+g_{12}\\ 
g_{21}A+g_{22} \\
 \end{array} \right | $$
 where  $g=\left ( \begin{array}{cc}
g_{11} & g_{12} \\
g_{21} & g_{22} \end{array} \right )$ with $g_{ij}\in M_{2}(\mathbb{R})$.
This defines local coordinates $(U_g,\Psi_g)$ on $Gr$ where $U_g=\Psi_g(M_{2}(\mathbb{R}))$.  The change of local coordinates from $U_g$ to $U_h$ is given by 
$$\Psi_{g,h}(A)=(k_{11}A+k_{12})(k_{21}A+k_{22})^{-1}$$ 
where $h^{-1}g=\left ( \begin{array}{cc}
k_{11} & k_{12} \\
k_{21} & k_{22} \end{array} \right )$ with $k_{ij}\in M_{2}(\mathbb{R})$. Note that the domain of  $\Psi_{g,h}$ is $U_{g,h}= \{ A | det(Ak_{12}+k_{22})\neq 0 \}$. \\
First we summarize some of   the properties of these changes of coordinates.\\
 (1)Tangent space: We identify the tangent space to $M_{2}(\mathbb{R})$ at any point with  $M_{2}(\mathbb{R})$ via $\sum_{ij}a_{ij}\frac{\partial}{\partial x_{ij}}\to
 \left( \begin{array}{cc}
a_{11} & a_{12} \\
a_{21} & a_{22}  \end{array} \right)$
Then it is easy to see that
\begin{lemma}
 $T_{X}\Psi_{g,h}:M_{2}(\mathbb{R})\to M_{2}(\mathbb{R})$ is given by
$$T_{X}\Psi_{g,h}(A)= (k_{11}-(k_{11}X+k_{12})(k_{21}X+k_{22})^{-1}k_{21})A(k_{21}X+k_{22})^{-1}$$
\end{lemma}
In particular this shows that  $Gr$  is orientable. We take the orientation on $Gr$ which in local coordinates $U_g$ is given by 
$$dx_{11}\wedge dx_{21}\wedge dx_{12}\wedge dx_{22}$$
(2) For a function $f:M_{2}(\mathbb{R})\to \mathbb{R}$ we define
 $$\frac{\partial f}{\partial X}= \left( \begin{array}{cc}
\frac{\partial f}{\partial x_{11}} & \frac{\partial f}{\partial x_{21}}\\
\frac{ \partial f}{\partial x_{12}} &  \frac{\partial f}{\partial x_{22}} \end{array} \right)$$
Then one can see that 
\begin{lemma}
For any  $f:M_{2}(\mathbb{R})\to \mathbb{R}$  we have,
$$\frac{\partial (f\circ \Psi_{g,h})}{\partial X}= (k_{21}X+k_{22})^{-1}
\frac{\partial f }{\partial X}\circ \Psi_{g,h}  (k_{11}-(k_{11}X+k_{12})(k_{21}X+k_{22})^{-1}k_{21})
$$
\end{lemma}
%In the same way we identify the cotangent space at any point of $M_{2}(\mathbb{R})$  with   $M_{2}(\mathbb{R})$ via $\sum_{ij}a_{ij}dx_{ij} \to
 %\left( \begin{array}{cc}
%a_{11} & a_{21} \\
%a_{12} & a_{22}  \end{array} \right)$.\\
(3) For a function $f:M_{2}(\mathbb{R})\to M_{2}(\mathbb{R})$, we define 
$$df:=\sum_{ij}\frac{\partial f}{\partial x_{ij}}dx_{ij}$$
as an $M_{2}(\mathbb{R})$-valued 1-form.  For example
$$dX= \left( \begin{array}{cc}
dx_{11} & dx_{12} \\
dx_{21} &dx_{22}  \end{array} \right)$$               
It is easy to see that  
\begin{lemma}
$$ \Psi_{g,h}^{*}(dX)=(k_{11}-(k_{11}X+k_{12})(k_{21}X+k_{22})^{-1}k_{21})dX (k_{21}X+k_{22})^{-1}.$$
\end{lemma}
(4) Conformal structure on $Gr$: Consider the following metric on $M_{2}(\mathbb{R})$,
$$ds^2=2(dx_{11}dx_{22}-dx_{12}dx_{21})=2 det(dX).$$
\begin{proposition}
The metric changes under the change of coordinates as
$$\Psi_{g,h}^{*}(ds^2)=\frac{1}{det(k_{21}X+k_{22})^2}ds^2$$
\end{proposition}
\begin{proof}
This follows from lemma 3.3 and the fact that
$$det(k_{21}X+k_{22})^{-1}=det(k_{11}-(k_{11}X+k_{12})(k_{21}X+k_{22})^{-1}k_{21})$$
\end{proof}
This proposition implies that we have a conformal structure on $Gr$ which on $U_g$ is given by the above metric.  Using this conformal structure and the orientation we can consider ASDYM equations on $Gr$ which in local coordinates are just split ASDYM equations. We need to have some information about vector bundles on $Gr$.  First of all we have the universal vector bundle on $Gr$. We denote the universal vector bundle on $Gr$ by $\mathcal{P}$ as before. It is the vector sub-bundle of the trivial bundle $Gr\times \mathbb{R}^4$ whose fiber at $p\in Gr$  is $p$ itself.
On each $U_g$, we have the following trivialization           
$$U_g\times \mathbb{R}^2\to  \mathcal{P}$$
$$ (p,\left( \begin{array}{cc}
x \\
 y \end{array} \right) )\to (p, g   \left( \begin{array}{cc}
A\\  
I
 \end{array}   \right)
 \left( \begin{array}{cc}
x \\
 y \end{array} \right)  
)
$$
 where $\Psi_g(A)=p$. Note that the transition functions of $\mathcal{P}$ are given by $k_{21}X+k_{22}$.\\
We denote $\wedge^{2}\mathcal{P}$ by $\varepsilon[-1]$. Hence the transition functions of  $\varepsilon[-1]$ are given by $det(k_{21}X+k_{22})$. We have a locally constant line bundle on $Gr$ whose transition functions are given by $\frac{det(k_{21}X+k_{22})}{|det(k_{21}X+k_{22})|}$. We denote this line bundle by $\tilde{\varepsilon}$.
 For any $n\in \mathbb{Z}$  and any vector bundle $E$ on $Gr$ we set $\varepsilon[n]=\varepsilon[-1]^{-n}$, $E[n]:=E\otimes \varepsilon[n]$ and $\widetilde{E}=E\otimes\tilde{\varepsilon}$.\\
On each $U_g$ we have the following second order differential operator 
$$\square_{2,2}=\frac{\partial^2}{\partial x_{11}\partial x_{22}}-\frac{\partial^2}{\partial x_{12}\partial x_{21}}$$
 One can check that under the change of local coordinates we have, see \cite{KO},
 $$\square_{2,2} (|det(k_{21}X+k_{22})|^{-1} f\circ \Psi_{g,h})=|det(k_{21}X+k_{22})|^{-3} \square_{2,2} ( f)\circ \Psi_{g,h}$$
 This means that we have a global differential operator from  $\widetilde{\varepsilon[-1]}$ to  $\widetilde{\varepsilon[-3]}$
which in each local coordinates $U_g$ is given by  $\square_{2,2} $. We also denote this global differential operator by $\square_{2,2} $.
\end{subsection}
 \begin{subsection}{ t'Hooft ansatz}                                                    
Using our notation we rewrite the famous t'Hooft ansatz.  For any smooth function $f: \mathbb{H}_{\mathbb{R}}\to \mathbb{R}$, we define  the local t'Hooft ansatz to be
 $$A(f)=-\frac{\frac{\partial f}{\partial X}}{f}dX$$
 which is an $\mathbb{H}_{\mathbb{R}}$-valued 1-form. Then we have the following famous result
 \begin{proposition}(Local t'Hooft ansatz)
The  curvature 
$$F(f)=dA(f)+A(f)\wedge A(f)$$
  of the connection potential $A(f)$ is ASD if  and only if $f$ satisfies
$$\square_{2,2}f:= \frac{\partial^2 f}{\partial x_{11}\partial x_{22}}-\frac{\partial^2 f}{\partial x_{12}\partial x_{21}}=0$$
\end{proposition} 
For the global picture, we need to know  how $A(f)$ is transformed under the change of coordinates. we have
\begin{proposition} \label{TA}
We have
$$\Psi_{g,h}^{*}(A(f))=(k_{21}X+k_{22})A(f\circ \Psi_{g,h})(k_{21}X+k_{22})^{-1}$$

\end{proposition}
We would like to obtain a global version of t'Hooft ansatz. Note that if $f:U\to \mathbb{R}$ is a smooth nowhere vanishing function on an open subset of  $M_2(\mathbb{R})$ then the t'Hooft ansatz $A(f)$ gives a smooth solution of $GL(2,\mathbb{R})$-ASDYM equations on $U$. However t'Hooft ansatz depends on the local coordinates. Nevertheless if we pass to appropriate line bundles we obtain a global version of the t'Hooft ansatz. More precisely
  \begin{theorem}
(1) (Global t'Hooft ansatz) There is a unique map 
 $$\phi:\Gamma(Gr,\widetilde{\varepsilon[-1]} )^*\to \mathcal{A}(\mathcal{P})$$
such that on local coordinates it is given by local t'Hooft ansatz. Here $\Gamma(Gr,\widetilde{\varepsilon[-1] })^*$ is the space of nowhere vanishing sections of 
 $\widetilde{\varepsilon[-1]} $ and $\mathcal{A}(\mathcal{P})$ is the space of connections on 
$\mathcal{P}$.\\
(2) For $s\in \Gamma(Gr,\widetilde{\varepsilon[-1]})^*$, the curvature of $\phi(s)$ is ASD if and only if 
$ \square_{2,2}s=0$.
  \end{theorem}
 \begin{proof}
First  we recall that  constructing a connection on a vector bundle $E$ with local trivializations $E|_{U_i}=U_i\times \mathbb{R}^n$ and transition functions 
$(U_i,\pi_{ij})$ is the same as giving $M_n(\mathbb{R})$-valued 1-forms $A_i$ on $U_i$ such that 
$$A_j=\pi_{ij}^{-1}A_i\pi_{ij}+\pi_{ij}^{-1}d\pi_{ij}$$
on $U_i\cap U_j$. \\
We denote the transition functions of $\mathcal{P}$ by $\pi_{gh}$.\\
  (1) We just need to check that the different local definitions of t'Hooft ansatz match to give a global map $\phi$. This is equivalent to proving that    we have
  $$ A(f_g)=\pi_{gh}^{-1}\Psi_{gh}^{*}(A(f_h))\pi_{gh}+\pi_{gh}^{-1}d\pi_{gh}$$
  when $f_h\circ\Psi_{g,h}=|det(k_{21}X+k_{22})|f_g$.
From proposition \cite{TA} we have,
 $$\Psi_{g,h}^{*}(A(f_h))=\pi_{gh}A(f_h\circ \Psi_{g,h})\pi_{gh}^{-1}$$
 So we need to prove that 
 $$A(|det(k_{21}X+k_{22})|)+\pi_{gh}^{-1}d\pi_{gh}=0\qquad (**)$$
 Clearly $A(|det(k_{21}X+k_{22})|)=A(det(k_{21}X+k_{22}))$ and 
 a simple computation shows that 
 $$\frac{\partial det(k_{21}X+k_{22})}{\partial X}=det(k_{21}X+k_{22}) (k_{21}X+k_{22})^{-1}k_{21}$$
 so 
 $$A(det(k_{21}X+k_{22}))=-(k_{21}X+k_{22})^{-1}k_{21}dX=-\pi_{gh}^{-1}d\pi_{gh}$$
which proves $(**)$.\\
(2) This follows form the local t'Hooft ansatz.
 
 \end{proof}
%  \begin{remark}Note that $SL(4,\mathbb{R})$ acts naturally on $\Gamma(Gr,\widetilde{\varepsilon[-1]} )^*$ and $\mathcal{A}(\mathcal{P})$. Moreover it is not hard to see that $\phi$ preserves this action.  \end{remark}
Therefore in order to produce ASDYM solutions on $\mathcal{P}$ we only need to start form a solution of $\square_{2,2}s=0$ on $\widetilde{\varepsilon[-1]}$. Fortunately there is a classical result which produce all the global solutions. More precisely,    
it is well-known that there is a bijection 
$$\Gamma(\mathbb{R}\mathbb{P}^3,\varepsilon(-2))\to ker(\square_{2,2}:\Gamma(Gr,\widetilde{\varepsilon[-1]})\to\Gamma(Gr,\widetilde{\varepsilon[-3]}))$$
This transform is known as the X-ray transform \cite{Ar,Sp}. One can Identify $\Gamma(\mathbb{R}\mathbb{P}^3,\varepsilon(-2))$ with homogeneous functions on $\mathbb{R}^4-\{0\}$ of degree $-2$. Now, the simplest homogeneous functions on $\mathbb{R}^4-\{0\}$ of degree $-2$ is $\frac{1}{x_1^2+x_2^2+x_3^2+x_4^2}$. Under the X-ray transform this function goes to a solution $f_0\in \Gamma(Gr,\widetilde{\varepsilon[-1]} )^*$of $\square_{2,2}$ which in local coordinates is given by 
$$f_0(X)=\frac{1}{\sqrt{det(1+X^tX)}}$$
see \cite{Sp}.  As we saw this solution of  $\square_{2,2}s=0$ gives a solution to the ASDYM equations on $Gr$. In local coordinates this solution is just $A(f_0)$.  A simple computation shows that
$$A(f_0)=\frac{-1}{2}A(det(1+X^tX))=(1+X^tX)^{-1}X^tdX$$
which is just the basic split anti-instanton.   We just summarize this discussion in the following corollary
\begin{corollary}
The t'Hooft ansatz applied to $\frac{1}{x_1^2+x_2^2+x_3^2+x_4^2}$ yields the basic split anti-instanton. Moreover all the split anti-instantons of topological charge $-1$ can be obtained via the t'Hooft ansatz.
\end{corollary}   
If we start with any positive quadric form $Q$ on $\mathbb{R}^4$, then $\frac{1}{Q(x)}$ gives a  global section of $ \varepsilon(-2)$. One can see that the t'Hooft ansatz applied to this section gives a split anti-instanton of topological charge $-1$. On the other hand $SL(4,\mathbb{R})$ acting on the standard positive quadric form on $\mathbb{R}^4$ gives all the positive quadratic forms and $SO(4)$ leaves it invariant. In summary, we see that the t'Hooft ansatz provides an $SL(4,\mathbb{R})$-equivariant isomorphism between the space of  of positive quadratic forms on $\mathbb{R}^4$ and the moduli space of split anti-instantons of topological charge $-1$.\\
Finally, we note that similar formulas give rise to the Euclidean t'Hooft ansatz. But there is a big difference between the two signatures. In the Euclidean case we have to start with singular solutions and the topological charge of the solution depends on the kind of singularity of the solution. In the split case, we start with global solutions and the topological charge is always $-1$.

\end{subsection}

\end{section}

\begin{section}{Split ADHM construction}
In the Euclidean case, the so-called ADHM construction gives all the solutions to $Sp(n)$-(A)SDYM equations. 
The similarity between our construction of the basic split instanton  and basic instanton suggests that the ADHM construction of multi-instantons must have a counterpart in the split case and in fact this is the case.  In this section we explain the split ADHM construction. \\  
First we recall the ADHM construction of multi-instantons of charge $k$ and structure  group $Sp(n)$, see \cite{At}. Given $n+k$ by $k$ quaternionic matrices $A$ and $B$ we consider the following $\mathbb{H}$-linear maps 
$$v(X,Y)=AX+BY:\mathbb{H}^k\to\mathbb{H}^{k+n}$$
Moreover we have  a non-degeneracy condition which is 
$$v(X,Y)\quad \text{has  maximal rank for any} \quad (X,Y)\neq(0,0)$$ 
This non-degeneracy condition implies that 
the co-kernels of these maps define a quaternionic vector bundle $E$ on $\mathbb{H}\mathbb{P}^1$. 
Consider the following quaternionic inner product on $\mathbb{H}^{k+n}$,
$$( \begin{pmatrix}
 a_1\\
 a_2\\
  \vdots\\
 a_{n+k}\\
  \end{pmatrix},\begin{pmatrix}
 b_1\\
 b_2\\
  \vdots\\
 b_{n+k}\\
  \end{pmatrix})=\bar{b}_1a_1+\bar{b}_2a_2+\dots+\bar{b}_{n+k}a_{n+k}.
  $$
Using this inner product we can consider the projection of the trivial connection onto $E$ where we identify $E$ with the orthogonal complement of the image bundle of $v(X,Y)$.
Then we have the following observation:
\newline
The projection of the trivial connection onto $E$ is ASD if and only if for any $X\in\mathbb{H}$
$$[(AX+B)^*(AX+B)]^{-1}$$
is real, where $A^*$ is defined by $(A^*)_{ij}=\bar{A}_{ji}$. This condition is equivalent to the condition that  $B$ and $A^*A+B^*B$ are symmetric as matrices over quaternions. Furthermore all the solutions are of this form. \\
The split ADHM construction is basically the same construction except the quaternionic conjugation operation is replaced by matrix transpose operation. 
More precisely, given  $n+k$ by $k$ split quaternionic matrices $A$ and $B$ we consider maps 
$$v(X,Y)=AX+BY:\mathbb{R}^{2k}\to\mathbb{R}^{2k+2n}$$
Note that here we identify an $m$ by $n$ matrix of split quaternions with a $2m$ by $2n$ matrix over real numbers under the realization $\mathbb{H}_{\mathbb{R}}=M_2(\mathbb{R})$.
We now impose the following  non-degeneracy condition
$$v(X,Y)\quad \text{has maximal rank for any} \quad [X:Y]\in\mathbb{H}_{\mathbb{R}}\mathbb{P}^1.$$ 
Then the images of $v(X,Y)$ which only depend on $ [X:Y]\in\mathbb{H}_{\mathbb{R}}\mathbb{P}^1$ define a vector subbundle of the trivial bundle on $\mathbb{H}_{\mathbb{R}}\mathbb{P}^1$.   Consider the usual inner product on $\mathbb{R}^{2n+2k}$ which is  
$$( \begin{pmatrix}
 a_1\\
 a_2\\
  \vdots\\
 a_{2n+2k}\\
  \end{pmatrix},\begin{pmatrix}
 b_1\\
 b_2\\
  \vdots\\
 b_{2n+2k}\\
  \end{pmatrix})=a_1b_1+a_2b_2+\dots+a_{2n+2k}b_{2n+2k}.
  $$
  We denote the orthogonal complement of the vector bundle defined by the images of $v(X,Y)$ by $E$.  Finally, we say $a\in \mathbb{H}_{\mathbb{R}}$ is split real if $a^t=a$
 \begin{proposition}
The projection of the trivial connection onto $E$ is an $O(2n)$-SD connection if and only if for any $X\in\mathbb{H}_\mathbb{R}$, 
$$[(AX+B)^T(AX+B)]^{-1} $$
is split real, i.e. its entries are split real quaternions. Here $A^{T}$ is defined as $(A^{T})_{ij}=A^t_{ji}$, in other words $A^T$ is just the transpose of $A$ considered as a $2n+2k$ by $2k$ real matrix. 
\end{proposition}
\begin{proof}
The proof is essentially the same as the Euclidean case, see \cite{At}. In fact if $u$ (which is considered to be a $2n+2k$ by $2n$ real matrix) is a local frame for $E$ and $v$ (which is considered to be a $2n+2k$ by $2k$ real matrix) is a local frame for the orthogonal complement of $E$ then one can see that the matrix of the curvature of the connection on $E$ in the local frame $u$ is given by
$$F=(u^Tu)^{-1}u^T(dv\wedge(v^Tv)^{-1}dv^T)u$$ 
Now on local coordinates $\{[X:1]|X\in \mathbb{H}_{\mathbb{R}}\}$ we can choose $v=AX+B$. Therefore
$$F=(u^Tu)^{-1}u^T(d(AX+B)\wedge(v^Tv)^{-1}d(X^tA^T+B^T))u$$
$$=(u^Tu)^{-1}u^TA(dX\wedge(v^Tv)^{-1}dX^t)A^Tu$$
This implies that $F$ is SD if and only if 
$$(v^Tv)^{-1}=[(AX+B)^T(AX+B)]^{-1} $$
 is split real for any $X$.

\end{proof}
A few remarks are in order. We note that the condition to obtain solutions is that $[(AX+B)^T(AX+B)]^{-1}$ is split real for any $X$. In contrast with the Euclidean case we cannot conclude that $[(AX+B)^T(AX+B)]$ is split real for every $X$ (except when $k=1$). It does not seem to be easy to simplify this condition to some conditions on $A$ and $B$ as in  the Euclidean case. Moreover, it is not clear if there are matrices $A$ and $B$ with this  property. Nevertheless, we note that, for $k=1$, this condition is always fulfilled. In particular, $n=k=1$, the above theorem yields all the basic instantons of topological charge $1$. \\
There are many questions remain to be answered here. Two main ones are
 (1) What is the moduli space of the solutions obtained via the split ADHM construction?\\
  (2) How can one characterize the solutions coming from the split ADHM construction?\\
The split ADHM   construction has a complex analog.  More precisely,  Given  $n+k$ by $k$ complex quaternionic matrices $A$ and $B$ we consider maps 
$$v(X,Y)=AX+BY:\mathbb{C}^{2k}\to\mathbb{C}^{2k+2n}$$
 Note that here we identify an $m$ by $n$ matrix of complex quaternions with a $2m$ by $2n$ matrix over complex numbers under the realization $\mathbb{H}_{\mathbb{C}}=M_2(\mathbb{C})$.
We now impose the following  non-degeneracy condition 
$$v(X,Y)\quad \text{has maximal rank for any} \quad [X:Y]\in\mathbb{H}_{\mathbb{C}}\mathbb{P}^1.$$ 
Then the images of $v(X,Y)$ which only depend on $ [X:Y]\in\mathbb{H}_{\mathbb{C}}\mathbb{P}^1$ define a holomorphic vector sub-bundle of the trivial bundle on $\mathbb{H}_{\mathbb{R}}\mathbb{P}^1$.  We consider the standard Hermitian product on $\mathbb{C}^{2k+2n}$. We denote the orthogonal complement of the vector bundle defined by the images of $v(X,Y)$ by $\mathcal{E}(A,B)$. Let $\nabla(A,B)$ be the projection of the trivial connection on $\mathbb{C}^{2k+2n}$ onto $\mathcal{E}(A,B)$. Computations similar to the above proposition shows that the curvature of $\nabla$ is of type $(1,1)$.  Therefore there is a holomorphic structure on $\mathcal{E}(A,B)$ defined by $\nabla^{0,1}$.  Therefore the complex ADHM construction gives rise to a family of holomorphic vector bundles equipped with a unitary structure.  We see that the holomorphic structure and the unitary structure uniquely determine $\nabla(A,B)$ because 
 it is well-known that if we have a unitary structure on a holomorphic vector bundle $V$ then there is a unique smooth connection $V$ which is compatible with both the holomorphic structure and unitary structure, see \cite{At}.  The importance of the complex picture is that it relates the Euclidean SDYM equations and split ASDYM equations.  Here we explain this connection for the case $n=k=1$ but we believe that a similar picture holds higher ranks and charges. 
 Suppose that $V$ is a holomorphic vector bundle on $\mathbb{H}_{\mathbb{C}}\mathbb{P}^1$. We denote the moduli space of unitary structures on $V$ by $\mathcal{U}_V$ where to unitary structure considered to be equivalent if there is a holomorphic isomorphism of $V$ sending one of them to the other one.  Let $\mathcal{P}$ be the universal vector bundle of $\mathbb{H}_{\mathbb{C}}\mathbb{P}^1$ and $\mathcal{Q}:=\mathbb{C}^4/\mathcal{P}$. Then
 we can embed the moduli space of anti-instantons of topological charge $1$ and the moduli space of split instantons of topological charge $-1$ into $\mathcal{U}_{\mathcal{Q}}$ in a natural way.  To see this, we just use the ADHM construction in all cases. \\
This already relates the (A)SDYM equations with the theory of holomorphic vector bundles on complex 4-manifolds. This also opens a way to find a finite dimensional moduli space of solutions in the split case. Even though $\mathcal{U}_V$ is infinite dimensional, it seems possible to find a finite dimensional moduli space of unitary structures in a natural way.   In any case, it seems that there is a deep relation between the Euclidean, split and Complex pictures which needs to be studied thoroughly.

\end{section}


\begin{thebibliography}{ZZZZ}
%\bibitem[Cohn]{Cohn}P.M. Cohn,
%{\it Skew Fields. Theory of General Division Rings}, Encyclopedia of mathematics and its application volume 57, Cambridge University Press, 1995.
\bibitem{Ar} Aryapoor,  M. The Penrose Transform in the Split Signature, preprint, math.AG/08123692
\bibitem{At} Atiyah, M.F. Geometry of Yang-Mills Fields, Scuola Normale Superiore Pisa, Pisa (1979) 
\bibitem{At1} Atiyah, M.F. $K$-theory, Lecture notes by D W Anderson, W A benjamin, New York-Amsterdam (1967) 
\bibitem{ADHM} Atiyah, M.;  Drinfeld, V.; Hitchin, N.; Manin, Yu. Construction of instantons, Phys. Lett. A 65 (1978) 185Ð187.
%\bibitem{AHS} Atiyah, M.F. ; Hitchin, N.J.;  Singer,  I.M. Self-duality in four-dimensional Riemannian geometry, Proc. Roy. Soc. London 362 (1978), 425-461.
\bibitem{Ch} Chern, S.S. On the multiplication in the characteristic ring of a sphere bundle, Ann. of Math. 
49 (1948) 362-372.  
\bibitem{DK}  Donaldson, S K;  Kronheimer, P. The Geometry of Four-manifolds, Oxford Univ. Press (1990) 
%\bibitem{Ea}  Eastwood, M.  Complex methods in real integral geometry. With the collaboration of T. N. Bailey and C. R. Graham. The Proceedings of the 16th Winter School ``Geometry and Physics'' (Srn', 1996). Rend. Circ. Mat. Palermo (2) Suppl. No. 46 (1997), 55--71.
\bibitem{Eh} Ehresmann, C. Sur La topologie des certaines varieties algebriques reelles, Journal de math. Pures, 104 1939, 69-100.  
\bibitem{FU}   Freed, D.S.;   Uhlenbeck, K.K. 
Instantons and Four-manifolds, MSRI Publications 1, Springer, 
New York (1984)
\bibitem{GS} Guillemin, V.; Sternberg, S.An ultra-hyperbolic analogue of the Robinson-Kerr theorem. Lett. Math. Phys. 12 (1986), no. 1, 1--6. 
\bibitem{KN} Kobayashi, S.; Nomizu, K. Foundation of  differential geometry, Vol. 2,  John Wiley 
\& Sons, Inc. Interscience Division, (New York, 1963)
\bibitem{KO}Kobayashi, T.; \O rsted, B.  Analysis on the minimal representation of $O(p,q)$. I. Realization via conformal geometry. Adv. Math. 180 (2003), no. 2, 486--512.
\bibitem{Ma}   Mason, L.J. Global anti-self-dual Yang-Mills fields in split signature and their scattering. 
J. Reine Angew. Math. 597 (2006), 105--133.
\bibitem{MW}  Mason, L.J. ;  Woodhouse, N.M.J. Integrability, Self-duality and Twistor Theory, Clarendon 
Press, Oxford, 1996.
\bibitem{Na}Nakajima, H. Lectures on Hilbert schemes of points on surfaces. University Lecture Series, 18. American Mathematical Society, Providence, RI, 1999
\bibitem{OS} Okonek, C. ; Schneider, M. ; Spindler, H. Vector bundles on complex projective spaces. Progress in Mathematics, 3. BirkhŠuser, Boston, Mass., 1980
 \bibitem{Sp}      Sparling, G.  
Inversion for the Radon line transform in higher dimensions. (English summary) 
R. Soc. Lond. Philos. Trans. Ser. A Math. Phys. Eng. Sci. 356 (1998), no. 1749, 3041--3086. 
 
 

\end{thebibliography}
\end{document}